\documentclass[11pt]{amsart}

\usepackage{amsmath,amsthm}
\usepackage{amssymb,amsfonts}
\usepackage{hyperref}
\usepackage{tikz}

\usepackage{stmaryrd}
\DeclareSymbolFont{bbold}{U}{bbold}{m}{n}
\DeclareSymbolFontAlphabet{\mathbbm}{bbold}

\title{Useful axioms}

\keywords{MSC 2010: 03E25, 03E30, 03E40, 03E57, 03C20, 03C90.}
\theoremstyle{plain}
	\newtheorem{theorem}{Theorem}[section]
	\newtheorem{proposition}[theorem]{Proposition}
	\newtheorem{lemma}[theorem]{Lemma}
	\newtheorem{corollary}[theorem]{Corollary}
	\newtheorem{fact}[theorem]{Fact}

\theoremstyle{definition}
	\newtheorem{definition}[theorem]{Definition}
	\newtheorem{notation}[theorem]{Notation}
	
	\newtheorem{notation*}{Notation}

\theoremstyle{remark}
	\newtheorem{remark}[theorem]{Remark}

\newcommand{\Ord}{\ensuremath{\mathrm{Ord}}}

\newcommand{\ZFC}{\ensuremath{\mathsf{ZFC}}}
\newcommand{\ZF}{\ensuremath{\mathsf{ZF}}}
\newcommand{\MK}{\ensuremath{\mathsf{MK}}}

\DeclareMathOperator{\dom}{dom}

\DeclareMathOperator{\crit}{crit}

\DeclareMathOperator{\trcl}{trcl}
\DeclareMathOperator{\val}{val}
\DeclareMathOperator{\St}{St}

\DeclareMathOperator{\Coll}{Coll}

\DeclareMathOperator{\RO}{\mathsf{RO}}
\DeclareMathOperator{\CLOP}{CLOP}

\newcommand{\bool}[1]{\mathsf{#1}}
\newcommand{\tow}[1]{\mathcal{#1}}

\newcommand{\pow}[1]{\mathcal{P}\left(#1\right)}
\newcommand{\Int}[1]{\text{Int}\left(#1\right)}
\newcommand{\Cl}[1]{\text{Cl}\left(#1\right)}
\newcommand{\Reg}[1]{\text{Reg}\left(#1\right)}
\newcommand{\rest}[1]{\upharpoonright_{#1}}
\newcommand{\cp}[1]{\left( #1 \right)}

\newcommand{\Qp}[1]{\left\llbracket #1 \right\rrbracket}
\newcommand{\ap}[1]{\langle #1 \rangle}
\newcommand{\bp}[1]{\left\lbrace #1 \right\rbrace}

\newcommand{\MM}{\ensuremath{\text{{\sf MM}}}} 
\newcommand{\SP}{\ensuremath{\text{{\sf SP}}}}
\newcommand{\SSP}{\ensuremath{\text{{\sf SSP}}}}

\newcommand{\FA}{\ensuremath{\text{{\sf FA}}}}

\theoremstyle{definition}

\newcommand{\RA}{\ensuremath{\text{{\sf RA}}}}

\newcommand{\UnderTilde}[1]{{\setbox1=\hbox{$#1$}\baselineskip=0pt\vtop{\hbox{$#1$}\hbox to\wd1{\hfil$\sim$\hfil}}}{}}

\newcommand{\PFA}{\ensuremath{\text{{\sf PFA}}}}

\newcommand{\BFA}{\ensuremath{\text{{\sf BFA}}}}

\newcommand{\AX}{\ensuremath{\text{{\sf AX}}}}

\newcommand{\NBG}{\ensuremath{\mathsf{NBG}}}

\DeclareMathOperator{\dd}{dd}
\DeclareMathOperator{\cpd}{cpd}

\usepackage[mode=multiuser,status=draft,lang=english]{fixme}
\fxsetup{theme=colorsig}

\FXRegisterAuthor{M}{aM}{M}
\FXRegisterAuthor{R}{aR}{Rapha\"{e}l}
\FXRegisterAuthor{S}{aS}{Silvia}

\author{Matteo Viale}
\date{}

\begin{document}
	\maketitle

%\listoffixmes
	
%	\input{Introduction}
%	\input{Notations}
\begin{abstract}
We give a brief survey of the interplay between forcing axioms and various other non-constructive 
principles widely used
in many fields of abstract mathematics, such as the axiom of choice and Baire's category theorem.

First of all we outline how, using basic partial order theory, 
it is possible to reformulate the axiom of choice, Baire's category theorem, and
many large cardinal axioms as specific instances of forcing axioms. We then address 
forcing axioms with a model-theoretic perspective and outline a deep analogy existing 
between the standard \L o\'s Theorem
for ultraproducts of first order structures and Shoenfield's absoluteness for $\Sigma^1_2$-properties.
Finally we address the question of whether and to what extent forcing axioms can provide
``complete'' semantics for set theory.
We argue that to a large extent this is possible for certain initial fragments of the universe of sets:
The 
pioneering work of Woodin on generic absoluteness show that this is the case for the 
Chang model $L(\Ord^\omega)$ (where all of mathematics formalizable in second order number theory
can be developed) in the presence of large cardinals, and recent 
works by the author with Asper\'o and with Audrito show that this can also be the case for the Chang model 
$L(\Ord^{\omega_1})$ (where one can develop most of mathematics formalizable in third 
order number theory) in the presence of large cardinals and maximal strengthenings of
Martin's maximum or of the proper forcing axiom.
A major question we leave completely open is whether this situation is peculiar to
these Chang models
or can be lifted up also to $L(\Ord^\kappa)$ for cardinals $\kappa>\omega_1$.
\end{abstract}	

% !TEX root = axiom-of-choice.tex

\section*{Introduction}
Since
its introduction by Cohen in 1963 forcing has been the key and the most effective 
tool for obtaining independence results in set theory.
This method has found applications in set theory and
in virtually all fields of pure mathematics: 
in the last forty years natural problems of group theory, functional analysis, 
operator algebras, general topology, and many other subjects were shown to 
be undecidable by means of forcing.
Starting from the early seventies and during the eighties it became transparent that 
many of these consistency results could all be derived by a short list of 
set theoretic principles, which are known in the literature as forcing axioms.
These axioms gave set theorists and mathematicians a very powerful tool to 
obtain independence results:
for any given mathematical problem we are most likely able to compute its (possibly different)
solutions 
in the constructible universe $L$ and in models of strong forcing axioms.
% (on the other hand in the case that
%the same solution of a given problem
%holds in $L$ and in models of forcing axioms, it is often the case that the possible consistency
%of a different solution to the same problem is much harder to find, or outward impossible).
These axioms settle basic problems in 
cardinal arithmetic like the size of the continuum and the singular 
cardinal problem (see among others the works of Foreman, 
Magidor, Shelah~\cite{foreman_magidor_shelah}, 
Veli\v{c}kovi\'c~\cite{VEL92}, Todorcevic~\cite{tod02}, 
Moore~\cite{moore.MRP}, Caicedo and Veli\v{c}kovi\'c~\cite{caivel06}, 
and the author~\cite{VIA08}), as well as combinatorially 
complicated ones like the basis problem for uncountable linear 
orders (see Moore's result~\cite{moore.basis} which extends 
previous work of Baumgartner~\cite{bau73}, Shelah~\cite{she76}, 
Todorcevic~\cite{tod98}, and others).
Interesting problems originating from other fields of mathematics and apparently 
unrelated to set theory have also been settled by appealing to forcing axioms, as is the case 
(to cite two of the most prominent examples)
for Shelah's
results~\cite{SHE74} on Whitehead's problem in group theory and Farah's result~\cite{FAR11}
on the non-existence of outer automorphisms of the Calkin algebra in operator algebra.
Forcing axioms assert that for a large class of compact topological spaces $X$
Baire's category theorem can be strengthened to the statement that any family of
$\aleph_1$-many dense open subsets of $X$ has non empty intersection.
In light of the success these axioms have met in solving problems
 a convinced platonist may start to argue that these principles
may actually give a ``complete'' theory of a suitable fragment of the universe of sets. 
However, it is not clear how one could formulate such a  result.
The aim of this paper is to explain in which sense we can show that
forcing axioms can give such a 
``complete'' theory and why they are so ``useful''. 

Section~\ref{sec:ACBCT} starts by showing that two basic non-constructive 
principles which play a crucial role in the foundations of many mathematical theories, the axiom of choice
and Baire's category theorem, can both be formulated as specific instances of forcing axioms.
In section~\ref{sec:LCFA} we also argue that many large cardinal axioms can be reformulated in the language of
partial orders as specific instances of a more general kind of forcing axiom. Sections~\ref{sec:LTGA} 
and~\ref{sec:CABCT} show that
Shoenfield's absoluteness for $\Sigma^1_2$-properties and \L o\'s Theorem for ultraproducts of first order models
are two sides of the same coin: 
%\L o\'s Theorem gives an equivalent formulation of the axiom of choice, 
%while Shoenfield's absoluteness can be proved appealing to Baire's category theorem. 
recast in the language of boolean valued models, Shoenfield's absoluteness shows
that there is a more general notion of boolean ultrapower 
(of which the standard ultrapowers encompassed by \L o\'s Theorem are just special cases) and that 
in the specific case in which one takes a boolean ultrapower of a compact, second countable space $X$, 
the natural embedding of $X$ in its boolean ultrapower
is at least $\Sigma_2$-elementary. Section~\ref{sec:MMWGA} 
embarks on a rough  analysis of what is a maximal forcing axiom. We are led by two driving observations, one rooted
in topological considerations and the other in model-theoretic arguments.
First of all we outline how Woodin's generic absoluteness results for  $L(\Ord^{\omega})$
entail that in the presence of large cardinals 
the natural embeddings of a separable compact Hausdorff space $X$ in its boolean ultrapowers
are not only $\Sigma_2$-elementary but fully elementary. 
We then present other recent results by the author, with Asper\'o \cite{VIAASP} and with Audrito \cite{VIAAUD14}
which show that, 
%the success met by forcing axioms
%such as Martin's maximum or the proper forcing axiom is also due to the fact that 
in the presence of natural 
strengthenings of Martin's maximum or of the proper forcing axiom, an exact analogue of Woodin's 
generic absoluteness result can be established also at the level of the Chang model $L(\Ord^{\omega_1})$
and/or for the first order theory of $H_{\aleph_2}$.
The main question left open is whether these generic absoluteness results are specific to the
Chang models $L(\Ord^{\omega_i})$  for $i=0,1$ or
can be replicated also for other cardinals.
The paper is meant to be accessible to a wide audience of mathematicians; specifically the first two sections 
do not require any special familiarity with logic or set theory other than some basic cardinal arithmetic. 
The third section requires a certain familiarity with first order logic and the basic model theoretic 
constructions of ultraproducts. The fourth and fifth sections, on the other hand, presume the reader has 
some familiarity with the forcing method. 

\newpage
\tableofcontents

\newpage
\section{The axiom of choice and Baire's category theorem as forcing axioms} \label{sec:ACBCT}

The axiom of choice $\mathsf{AC}$ and Baire's category theorem $\mathsf{BCT}$
are non-constuctive principles which 
play a prominent role in the development of many fields of abstract mathematics.
%, a few 
%fundamental samples of uses of the axiom of choice which are central in the development of the bare bones
%of mathematical theories
%are given by the use of Zorn's Lemma (an equivalent formulation of the axiom of choice)
%in algebra to produce maximal ideals of any given ring, and in functional analysis to produce
%a continuous extension to a given Banach space of any continuous linear functional which is defined
%on a closed proper subspace. Baire's category theorem is also of central importance: one key example is
%the open mapping theorem stating that any surjective bounded linear operator between Banach spaces is
%an open maps 
%
%In both cases the axiom of choice is used to produce by non constructive means
%an existential witness to a property.
%
%On the other hand Baire's ctegory theorem
Standard formulations of the axiom of choice and of Baire's category theorem are the following:
\begin{definition}
$\mathsf{AC}\equiv\,\,$ 
$\prod_{i\in I}A_i$ is non-empty for all families of non empty sets $\bp{A_i:i\in I}$, i.e. there is a choice function
$f:I\to\bigcup_{i\in I} A_i$ such that $f(i)\in A_i$ for all $i\in I$.
\end{definition}

\begin{theorem}\label{thm:BCTtopsp}
$\mathsf{BCT}_0\equiv$ For all compact Hausdorff spaces $(X,\tau)$ and all countable families
$\bp{A_n:n\in\mathbb{N}}$ of dense open subsets of $X$, $\bigcap_{n\in\mathbb{N}}A_n$ is non-empty.
\end{theorem}

%It is evident that both definitions assert the existence of existential witnesses for a given mathematical property
%for which we can produce just approximations
There are large numbers of equivalent formulations of the axiom of choice and it may come as a surprise that one 
of these is a natural generalization of Baire's category theorem and naturally leads to the notion of forcing axiom.

\begin{definition}
$(P,\leq)$ is a \emph{partial order} if $\leq$ is a reflexive and transitive relation on $P$.
\end{definition}

%The terminology comes from the observation that the downward closed subsets
%of a partial order
%$(P,\leq)$ define a topology on $P$ whose dense sets are exactly those satisfying the above property.

\begin{notation}
Given a partial order $(P,\leq)$, 
\[
\uparrow A=\bp{p\in P:\exists q\in A: q\leq p}
\]
denotes the \emph{upward closure} of $A$ and similarly $\downarrow A$ will denote its \emph{downward closure}.
\begin{itemize}
\item
$A\subseteq P$ is \emph{open} if it is a downward closed subset of $P$.
\item
The \emph{order topology} $\tau_P$ on $P$ is given by the downward closed subsets of $P$.
\item
$D$ is \emph{dense} if for all
$p\in P$ there is some $q\in A$ refining $p$ ($q$ refines $p$ if $q\leq p$),
\item
$G\subseteq P$ is a \emph{filter} if it is upward closed and all $q,p\in G$ have a common refinement $r\in G$.
\item 
$p$ is \emph{incompatible} with $q$ ($p\perp q$) if no $r\in P$ refines both $p$ and $q$.
\item
$X$ is a \emph{predense} subset of $P$ if $\downarrow X$ is open dense in $P$.
\item
$X$ is an \emph{antichain} of $P$ if it is composed of pairwise incompatible elements, and a maximal one if it is also predense.
\item
$X$ is a \emph{chain} of $P$ if $\leq$ is a total order on $X$.
 \end{itemize}
\end{notation}

The terminology for open and dense subsets of $P$ comes from the observation that the collection $\tau_P$ 
of downward closed subsets of $P$ is a topology on the space of points 
$P$ (though in general not a Hausdorff one), whose dense sets are exactly those satisfying the above property.
Notice also that the downward closure of a dense set is open dense in this topology.

A simple proof of the Baire
Category Theorem is given by a basic enumeration argument
(which however needs some amount of the axiom of choice to be carried):
\begin{lemma}\label{lem:BCTposets}
$\mathsf{BCT}_1\equiv$
Let $(P,\leq)$ be a partial order and $\bp{D_n:n\in\mathbb{N}}$ be a family of predense subsets of $P$.
Then there is a filter
$G\subseteq P$ meeting all the sets $D_n$.
\end{lemma}
\begin{proof}
Build by induction a decreasing chain $\bp{p_n:n\in\mathbb{N}}$ with
$p_{n}\in {\downarrow}\,D_n$ and $p_{n+1}\leq p_n$ for all $n$.
Let $G={\uparrow}\bp{p_n:n\in\mathbb{N}}$. Then $G$ meets all the $D_n$.
\end{proof}

Baire's category theorem can be proved from the above Lemma 
(without any use of the axiom of choice)
as follows:
\begin{proof}[Proof of $\mathsf{BCT}_0$ from $\mathsf{BCT}_1$]
Given a compact Hausdorff space $(X,\tau)$ and a family of dense open sets
$\bp{D_n:n\in\mathbb{N}}$ of $X$, consider the partial order $(\tau\setminus\bp{\emptyset},\subseteq)$
and the family $E_n=\bp{A\in \tau: \Cl{A}\subseteq D_n}$.
Then it is easily checked that each $E_n$ is dense open in the order topology induced by the partial order  
$(\tau\setminus\bp{\emptyset},\subseteq)$.
By Lemma~\ref{lem:BCTposets}, 
we can find a filter $G\subseteq \tau\setminus\bp{\emptyset}$ meeting all the sets $E_n$.
This gives that for all $A_1,\dots A_n\in G$
\[
\Cl{A_1}\cap \ldots\cap\Cl{A_n}\supseteq A_1\cap \ldots\cap A_n\supseteq B\neq\emptyset
\]
for some $B\in G$ (where $\Cl{A}$ is the closure of $A\subseteq X$ in the topology $\tau$.)
By the compactness of $(X,\tau)$,
\[
\bigcap\bp{\Cl{A}:A\in G}\neq\emptyset.
\]
Any point in this intersection belongs to the intersection of all the open sets $D_n$.
\end{proof}
Notice the interplay between the order topology on the partial order $(\tau\setminus\bp{\emptyset},\subseteq)$
and the compact topology $\tau$ on $X$. Modulo the prime ideal theorem (a weak form of the axiom of choice),
$\mathsf{BCT}_1$ can also be proved from 
$\mathsf{BCT}_0$. 

It is less well-known that the axiom of choice has also an equivalent formulation as the
existence of filters on posets meeting sufficiently many dense sets.
In order to proceed further, we need to introduce the standard notion of forcing axiom.
%The following argument was handed to me by Todor\v{c}evi\`c.
\begin{definition}\label{def:forcaxiom}
Let $\kappa$ be a cardinal and $(P,\leq)$ be a partial order.
\begin{quote}
$\FA_\kappa(P)\equiv$ \emph{For all families $\bp{D_\alpha:\alpha<\kappa}$ of predense
subsets of $P$, there is a filter $G$ on $P$ meeting all these predense sets.}
\end{quote}
Given a class $\Gamma$ of partial orders 
$\FA_\kappa(\Gamma)$ holds if 
$\FA_\kappa(P)$ holds for all $P\in \Gamma$. 
\end{definition}

\begin{definition}\label{def:<lambdaclos}
Let $\lambda$ be a cardinal.
A partial order $(P,\leq)$ is \emph{$<\lambda$-closed} if every decreasing
chain $\bp{P_\alpha:\alpha<\gamma}$ indexed by some $\gamma<\lambda$ has a lower bound in $P$. 

$\Gamma_\lambda$ denotes the class of $<\lambda$-closed posets.
$\Omega_\lambda$ denotes the class of posets $P$ for which $\FA_\lambda(P)$ holds.
\end{definition}

It is almost immediate to check that $\Gamma_{\aleph_0}$ is the class of all posets, and that
$\mathsf{BCT}_1$ states that $\Omega_{\aleph_0}=\Gamma_{\aleph_0}$.
The following formulation of the axiom of choice in terms of forcing axioms
was handed to me by Todorcevic, I'm not aware of 
any published reference. In what follows, let $\ZF$ denote the standard first order axiomatization 
of set theory in the first order language $\bp{\in,=}$
(excluding the axiom of choice) and $\ZFC$ denote $\ZF+$ the first order formalization of the axiom of choice.
\begin{theorem}\label{thm:forcax-AC-BCT}
The axiom of choice 
%\begin{itemize}
%\item
$\mathsf{AC}$ is equivalent (over the theory $\mathsf{ZF}$)
to the assertion that $\FA_\kappa(\Gamma_\kappa)$ holds for all regular
cardinals $\kappa$.
%\item
%$\mathsf{BCT}_1$ is equivalent (over the theory $\mathsf{ZF}$) to the assertion that all countable family
\end{theorem}

%\Rnote*{redundancy with end of section 2}{The above characterizations of $\mathsf{AC}$ and $\mathsf{BCT}_1$ show 
%that forcing axioms provide a common framework expressible in the language of partial orders
%where two (apparently unrelated) relevant
%principles of non-constructive mathematics can be formulated as specific instances of a simply definable family of properties. We will see over these notes that this is a feature of forcing axioms which is shared by many other
%non-constructive principles widely used in mathematics (or at least in set theory).}

We sketch a proof of Theorem~\ref{thm:forcax-AC-BCT}, the interested reader can find a full proof in 
\cite[Chapter 3, Section 2]{parente-tesi}
(see the following hyperlink:~\emph{\href{http://www.logicatorino.altervista.org/matteo_viale/thesis-parente.pdf}{Tesi-Parente}}).
First of all, it is convenient to prove~\ref{thm:forcax-AC-BCT} using a different equivalent formulation
of the axiom of choice.
\begin{definition} 
Let $\kappa$ be an infinite cardinal. 
The \emph{principle of dependent choices} $\mathsf{DC}_\kappa$ 
states the following: 
\begin{center}
For every non-empty set $X$ and every function $F\colon X^{<\kappa}\to\pow{X}\setminus\{\emptyset\}$, 
there exists $g\colon\kappa\to X$ such that $g(\alpha)\in F(g\restriction\alpha)$ for all $\alpha<\kappa$.
\end{center}
%\item
%$\mathsf{AC}_\kappa$ states that any family of size $\kappa$ of non-empty sets has a
\end{definition}

\begin{lemma}
$\mathsf{AC}$ is equivalent to $\forall\kappa\, \mathsf{DC}_\kappa$ modulo $\ZF$.
\end{lemma}
The reader can find a proof in \cite[Theorem 3.2.3]{parente-tesi}.
We prove the Theorem assuming the Lemma:
\begin{proof}[Proof of Theorem \ref{thm:forcax-AC-BCT}.]
We prove by induction on $\kappa$ 
that $\mathsf{DC}_\kappa$ is equivalent to 
$\FA_{\kappa}(\Gamma_{\kappa})$ over the theory $\mathsf{ZF}+\forall\lambda<\kappa\,\mathsf{DC}_\lambda$.
We sketch the ideas for the case $\kappa$-regular\footnote{In this case the assumption 
$\forall\lambda<\kappa\,\mathsf{DC}_\lambda$ is not needed, but all the relevant ideas in the 
proof of the equivalence are already present.}:

Assume $\mathsf{DC}_{\kappa}$; we prove (in $\ZF$) that $\mathsf{FA}_{\kappa}(\Gamma_\kappa)$ holds. 
Let $(P,\leq)$ be a ${<}\kappa$-closed partially ordered set, and 
$\bp{D_\alpha : \alpha<\kappa}\subseteq\pow{P}$ a family of predense subsets of $P$. 

Given a sequence $\ap{p_\beta : \beta<\alpha}$ call $\xi_{\vec{p}}$ the least $\xi$ such that 
$\ap{p_\beta : \xi\leq\beta<\alpha}$ is a decreasing chain if such a $\xi$ exists, and fix 
$\xi_{\vec{p}}=\alpha$ otherwise.
Notice that when the length $\alpha$ of $\vec{p}$ is successor then $\xi_{\vec{p}}<\alpha$.

%Say that $\ap{p_\beta : \beta<\alpha}$ is an almost-decreasing chain if for some $\xi<\alpha$
%$\ap{p_\beta : \xi\leq\beta<\alpha}$ is a chain antiisomorphic to $\alpha$. 
%For each almost-chain $\vec{p}=\ap{p_\beta : \beta<\alpha}$, let
%$\xi_{\vec{p}}$ be the least $\xi$ such that $\ap{p_\beta : \xi\leq\beta<\alpha}$ is a chain
%antiisomorphic to $\alpha-\xi$,
%and notice that 
%any $\ap{p_\beta : \xi\leq\beta<\alpha}$ of length a successor ordinal is an almost-decreasing chain.

We now define a function 
$F\colon P^{<\kappa}\to\pow{P}\setminus\{\emptyset\}$ as follows: given $\alpha<\kappa$ and a sequence
$\vec{p}\in P^{<\kappa}$, 
%If $\bp{p_\beta : \beta<\alpha}$ has limit length and is not an
%almost-chain, define $F(\ap{p_\beta : \beta<\alpha})=\bp{p_0}$. Otherwise, if 
%$\vec{p}=\bp{p_\beta : \beta<\alpha}$ is an almost-chain, let
\[
F(\vec{p})=\begin{cases}
\{p_0\}&\mbox{if }\xi_{\vec{p}}=\alpha\\
\bp{d\in {\downarrow}\,D_{\alpha} : d\le p_\beta\text{ for all } \xi_{\vec{p}}\leq\beta<\alpha}&\mbox{otherwise.}
\end{cases}
\]
The latter set is non-empty since $(P,\leq)$ is ${<}\kappa$-closed, $\alpha<\kappa$, 
and $D_\alpha$ is predense. By $\mathsf{DC}_\kappa$, 
we find $g\colon \kappa\to P$ such that $g(\alpha)\in F(g\restriction\alpha)$ for all $\alpha<\kappa$. 
An easy induction shows that for all $\alpha$ the sequence $g\restriction\alpha$ is decreasing, so
$g(\alpha)\in {\downarrow} D_{\alpha}$ for all $\alpha<\kappa$.
Then 
\[G=\bp{p\in P : \text{there exists }\alpha<\kappa\text{ such that }g(\alpha)\le p }\] 
is a filter on $P$, 
such that $G\cap D_\beta\neq\emptyset$ for all $\beta<\kappa$.

Conversely, assume $\FA_\kappa(\Gamma_\kappa)$,
we prove (in $\ZF$) that $\mathsf{DC}_\kappa$ holds. 

Let $X$ be a non-empty set and $F\colon X^{<\kappa}\to\pow{X}\setminus\{\emptyset\}$. 
Define the partially ordered set
\[
P=\bp{s\in X^{<\kappa} : \text{for all }\alpha\in\dom(s),\ s(\alpha)\in F(s\restriction\alpha)},
\]
with $s\le t$ if and only if $t\subseteq s$. Let $\lambda<\kappa$ and let $s_0\ge s_1\ge\dots\ge s_\alpha\ge\dots$, 
for $\alpha<\lambda$, be a chain in $P$. Then $\bigcup_{\alpha<\lambda}s_\alpha$ is clearly a lower bound for 
the chain. Since $\kappa$ is regular, we have $\bigcup_{\alpha<\lambda}s_\alpha\in P$ and so $P$ is 
${<}\kappa$-closed. For every $\alpha<\kappa$, define
\[
D_\alpha=\bp{s\in P : \alpha\in\dom(s)},
\]
and note that $D_\alpha$ is dense in $P$. Using $\FA_\kappa(\Gamma_\kappa)$, there exists a filter 
$G\subset P$ such that $G\cap D_\alpha\neq\emptyset$ for all $\alpha<\kappa$. Then 
$g=\bigcup G$ is a function $g\colon\kappa\to X$ such that $g(\alpha)\in F(g\restriction\alpha)$ for all 
$\alpha<\kappa$. 
\end{proof}

\section{Large cardinals as forcing axioms}\label{sec:LCFA}

%In this and the next sections it is more convenient to
From now on, we focus on boolean algebras rather than posets.
 
 \subsection{A fast briefing on boolean algebras}
 \begin{definition}
 A \emph{boolean algebra} $\bool{B}$ is a boolean ring i.e. a ring in which every element is idempotent.
 Equivalently a boolean algebra is a complemented distributive lattice $(\bool{B},\wedge,\vee,\neg,0,1)$
 (see \cite{GIVANTHALMOS}). 
 \end{definition}
  
 \begin{notation}
Given a boolean algebra $(\bool{B},\wedge,\vee,\neg,0,1)$, the poset $(\bool{B}^+;\leq_{\bool{B}})$ 
is given by its non-zero elements, with order relation given by $b\leq_{\bool{B}} q$ iff 
$b\wedge q=b$ iff $b\vee q=q$.

A boolean ring $(\bool{B},+,\cdot,0,1)$ has a natural structure of complemented distributive lattice
$(\bool{B},\wedge,\vee,\neg,0,1)$,
for which the sum on the boolean ring becomes the operation
 $\Delta$ of symmetric difference ($a\Delta b=a\vee b\wedge (\neg(a\wedge b))$) on the
 complemented distributive lattice, and the multiplication of the ring the operation $\wedge$.

We refer to filters, antichains, dense sets, predense sets, open sets on $\bool{B}$, meaning 
that these notions are declined for the corresponding partial order $(\bool{B}^+;\leq_{\bool{B}})$. 

We also recall the following:
%\begin{notation}
%Given a boolean algebra $\bool{B}$
\begin{itemize}
\item
An \emph{ideal} $I$ on $\bool{B}$ is a non-empty downward closed subset of $\bool{B}$ with respect to 
$\leq_{\bool{B}}$ which is also closed under $\vee$ (equivalently it is an ideal on the boolean ring 
$(\bool{B},\Delta,\wedge,0,1)$). Its \emph{dual filter} $\breve{I}$ is the set $\bp{\neg a:a\in I}$. It is 
a filter on the poset $(\bool{B}^+;\leq_{\bool{B}})$ (equivalently $I$ is an ideal in the boolean ring $\bool{B}$).
\item
An ideal $I$ on $\bool{B}$ is \emph{$<\delta$-complete} ($\delta$-complete) if all the subsets of $I$ of size less than 
$\delta$ (of size $\delta$) have an upper bound in $I$. 
\item
A \emph{maximal} ideal $I$ is an ideal properly contained in $\bool{B}$ and maximal with respect to this property
(equivalently it is a prime ideal on the boolean ring 
$(\bool{B},\Delta,\wedge,0,1)$).
Its dual filter is an \emph{ultrafilter}. An ideal $I$ is maximal if and only if $a\in I$ or $\neg a\in I$ for all $a\in\bool{B}$.
\item
$\bool{B}$ is $<\delta$-complete ($\delta$-complete) if all subsets of size less than $\delta$ (of size $\delta$)
have a supremum and an infimum.
\item
Given an ideal $I$ on $\bool{B}$, $\bool{B}/I$ is the quotient boolean algebra given by equivalence classes
$[a]_I$ obtained by $a=_I b$ iff $a\Delta b\in I$. %($a\Delta b=a\vee b\wedge (\neg(a\wedge b))$ 
%is the operation of symmetric difference).
\item 
$\bool{B}/I$ is $<\kappa$-complete if $I$ and $\bool{B}$ are both $<\kappa$-complete.
\item
$\bool{B}$ is \emph{atomless} if there are no minimal
elements in the partial order $(\bool{B}^+;\leq_{\bool{B}})$.
\item $\bool{B}$ is \emph{atomic} if the set of minimal
elements in the partial order $(\bool{B}^+;\leq_{\bool{B}})$ is open dense.
\end{itemize}
\end{notation}

%Forcing axioms were formulated by the requirement that the size of a given family 
%is the criteria by which one decides whether all the dense sets in the family can be met by a single filter.
%On the other hand we can devise more flexible criteria which (if met) implies the existence of a filter meeting a
%family of predense sets meeting these criteria. It is useful and simpler to state these criteria for partial orders of the
%form $(\bool{B}^+;\leq_{\bool{B}})$ for some boolean algebra $\bool{B}$.

Usually we insist in the formulation of forcing axioms 
on the requirement that for certain partial orders
$P$ any family of predense subsets of $P$ of some fixed size $\kappa$ can be met in a single filter.
In order to obtain a greater variety of forcing axioms, we need to consider a much richer variety 
of properties which characterizes the families of predense sets of $P$ which can be met in a single filter. 
Using boolean algebras, 
by considering partial orders of the
form $(\bool{B}^+;\leq_{\bool{B}})$ for some boolean algebra $\bool{B}$, 
we can formulate (using the algebraic structure of $\bool{B}$)
a wide spectrum of properties each defining a distinct forcing axiom.

\subsection{Measurable cardinals}

A cardinal $\kappa$ is measurable 
if and only if there is a uniform $<\kappa$-complete
ultrafilter on the boolean algebra $\pow{\kappa}$. The requirement that $G$ is uniform amounts to sayinging
that $G$ is disjoint from the ideal $I$ on the boolean algebra 
$(\pow{\kappa},\cap,\cup,\emptyset,\kappa)$ given by the bounded subsets of $\pow{\kappa}$. 
This means that we are actually looking for an ultrafilter $G$ on the boolean algebra
$\pow{\kappa}/I$. This is an atomless boolean algebra which is $<\kappa$-complete.
The requirement that $G$ is $<\kappa$-complete amounts to asking that $G$ selects an unique member of any partition of
$\kappa$ in $<\kappa$-many pieces, moreover any maximal antichain $\bp{[A_i]_I:i<\gamma}$ in the boolean algebra
$\pow{\kappa}/I$ of size $\gamma$  less than $\kappa$ 
is induced by a partition of $\kappa$ in $\gamma$-many pairwise
disjoint pieces.

All in all, we have the following characterization of measurability:
\begin{definition}
$\kappa$ is a \emph{measurable} cardinal if and only if there is an ultrafilter $G$ on $\pow{\kappa}/I$ 
(where $I$ is the ideal of bounded subsets of $\kappa$) which meets all the maximal antichains on
$\pow{\kappa}/I$ of size less than $\kappa$.
\end{definition}

In particular the measurability of $\kappa$ holds if and only if 
$(\pow{\kappa}/I)^+$ satisfies a certain forcing axiom stating that certain collections of predense subsets of 
$(\pow{\kappa}/I)^+$ can be simultaneously met in a filter.

We are led to the following definitions:
\begin{definition}
Let $(P,\leq)$ be a partial order and $\mathcal{D}$ be a family of non-empty subsets of $P$.
A filter $G$ on $P$ is $\mathcal{D}$-generic if $G\cap D$ is non-empty for all $D\in\mathcal{D}$.

Let $\phi(x,y)$ be a property and $(P,\leq)$ a partial order.
$\FA_\phi(P)$ holds if for any family $\mathcal{D}$ of predense subsets of $P$ such that
$\phi(P,\mathcal{D})$ holds there is some $\mathcal{D}$-generic filter $G$ on $P$.
\end{definition}

For instance, $\FA_\kappa(P)$ says that $\FA_\phi(P)$ holds for $\phi(x,y)$ being the property:
\begin{quote}
``$x$ is a partial order and $y$ is a family of predense subsets of $x$ of size $\kappa$'' 
\end{quote}
The measurability of $\kappa$ amounts to saying that $\FA_\phi(P)$ holds with $\phi(x,y)$
being the property
\begin{quote}
``$x$ is the partial order $(\pow{\kappa}/I)^+$ and $y$ is the (unique) family of predense subsets of $x$ consisting
of maximal antichains of  $(\pow{\kappa}/I)^+$ of size less than $\kappa$''
\end{quote}
We do not want to expand further on this topic but many other large cardinal properties of a 
cardinal $\kappa$ can 
be formulated as axioms of the form $\FA_\phi(P)$ for some property $\phi$ (for example this is the case for 
supercompactness, hugeness, almost hugeness, strongness, superstrongness, etc....).

In these first two sections we have already shown 
that the language of partial orders can accomodate 
three completely distinct and apparently unrelated families of
non-constructive principles which are essential tools in the development of many mathematical theories
(as is the case for the axiom of choice and of Baire's category theorem) and of crucial importance 
in the current developments of set theory (as is the case for large cardinal axioms).

\section{Boolean valued models, \L o\'s theorem, and generic absoluteness} \label{sec:LTGA}

We address here the correlation between forcing axioms and generic absoluteness results. We 
show how Shoenfield's absoluteness for $\Sigma^1_2$-properties and \L o\'s Theorem are two sides of 
the same coin: more precisely they are distinct specific cases of a unique general theorem which follows from $\mathsf{AC}$. 
%in particular we outline the dependece of \L o\'s Theorem on the axiom of choice and of Shoenfield's absoluteness on
%$\mathsf{BCT}_1$. 

After recalling the basic formulation of \L o\'s Theorem for ultraproducts, we introduce boolean valued models, 
and we argue that
\L o\'s Theorem for ultraproducts is the specific instance for complete atomic boolean algebras
of a more general theorem 
which applies to a much larger class of boolean valued models.
Then we introduce the concept of boolean ultrapower of a first order structure
on a Polish space $X$ endowed with Borel predicates $R_1,\dots,R_n$, 
and show that Shoenfield's absoluteness for $\Sigma^1_2$-properties
amounts to saying that the boolean ultrapower of $\ap{X,R_1,\dots,R_n}$ by any complete boolean algebra is a
$\Sigma_2$-elementary superstructure of $\ap{X,R_1,\dots,R_n}$.

\subsection{\L o\'s Theorem}
%\Rnote*{unnecessary}{ \L o\'s Theorem is a fundamental cornerstone in the development of model theory and of 
% compactness arguments in mathematical logic.}
  
 \begin{theorem}\label{thm:LOSstultrarpod}
 Let $\bp{\mathfrak{M}_l:l\in L}$ be models in a given first order signature 
 \[
 \mathcal{L}=\bp{R_i:i\in I,\,f_j:j\in J,\,c_k:k\in K },
 \]
 i.e. each $\mathfrak{M}_l=(M_l,R^l_i:i\in I,\,f^l_j:j\in J,\,c^l_k:k\in K)$.
 Let $G$ be an ultrafilter on $L$ (i.e. its dual is a prime ideal on the boolean algebra $\pow{L}$).
 Let 
 \[
 [f]_{G}=\bp{g\in\prod_{l\in L}M_l:\bp{l\in L: g(l)=f(l)}\in G}
 \]
 for each $f\in \prod_{l\in L}M_l$, and set
 \[
\prod_{l\in L}M_l/G=\bp{[f]_G:f\in \prod_{l\in L}M_l}.
 \]
For each $i\in I$ let $\bar{R}_i([f_1]_G,\dots,[f_n]_G)$ hold on $\prod_{l\in L}M_l/G$ if and only if
 \[
 \bp{l\in L:\,\mathfrak{M}_l\models R^l_i(f_1(l),\dots,f_n(l))}\in G.
 \]
Similarly interpret $\bar{f}_j:\prod_{l\in l}(M_l/G)^n\to\prod_{l\in L}M_l/G$ 
 and $\bar{c}_k\in \prod_{l\in l}M_l^n/G$ for each $j\in J$ and $k\in K$.
 
 Then: 
 \begin{enumerate}
 \item
 For all formulae $\phi(x_1,\dots,x_n)$ in the signature $\mathcal{L}$
 \[
 (\prod_{l\in L}M_l/G,\bar{R}_i:i\in I,\,\bar{f}_j:j\in J,\,\bar{c}_k:k\in K)\models \phi([f_1]_G,\dots,[f_n]_G)
 \]
 if and only if
 \[
 \bp{l\in L:\,\mathfrak{M}_l\models \phi(f_1(l),\dots,f_n(l))}\in G.
 \]
 \item
 Moreover if $\mathfrak{M}_l=\mathfrak{M}$ for all $l\in L$ (i.e. $\prod_{l\in L}M_j/G$ is the ultrapower of 
 $M$ by $G$), we have that
 the map
 $m\mapsto[c_m]_G$ (where $c_m:L\to M$ is constant with value $m$) defines an elementary embedding.
 \end{enumerate}
 \end{theorem}
 
 It is a useful exercise to check that the axiom of choice is essentially used in the induction step for 
 existential quantifiers in the proof of
 \L o\'s Theorem. Moreover \L o\'s Theorem is clearly a strengthnening of the axiom of choice, 
 for the very existence of an 
 element in $\prod_{l\in L}M_l/G$ implies that $\prod_{l\in L}M_l$ is non-empty.

 One peculiarity of  the above formulation of \L o\'s theorem is that it applies just to ultrafilters on $\pow{X}$. 
 We aim to find a ``most'' general formulation of this Theorem, which makes sense also for other kind of 
 ``ultraproducts'' and of ultrafilters on boolean algebras other than $\pow{X}$. This forces us to introduce
 the boolean valued semantics.
 
 \subsection{A fast briefing on complete boolean algebras and Stone duality}

 Recall that for a given topological space $(X,\tau)$ the regular open sets are those $A\in\tau$ such that
 $A=\Reg{A}=\Int{\Cl{A}}$ ($A$ coincides with the interior of its closure) and that $\RO(X,\tau)$ is the 
 complete boolean algebra whose elements are regular open
 sets and whose operations are given by
 $A\wedge B=A\cap B$, $\bigvee_{i\in I}A_i=\Reg{\bigcup_{i\in I}A_i}$, $\neg A=X\setminus\Cl A$.
 
 For any partial order $(P,\leq)$ the map $i:P\to \RO(P,\tau_P)$ given by 
 $p\mapsto\Reg{\downarrow\bp{p}}$ is 
 order and incompatibility preserving and embeds $P$ as a dense
 subset of the non-empty regular open sets in $\RO(P,\tau_P)$. 
 
 Recall also that the Stone space
 $\St(\bool{B})$ of a boolean algebra $\bool{B}$ is given by its ultrafilters $G$ and it is endowed with a compact
 topology $\tau_{\bool{B}}$ whose clopen sets are the sets $N_b=\bp{G\in\St(\bool{B}):b\in G}$ so that 
 the map $b\mapsto N_b$ defines a natural isomorphism of $\bool{B}$ with the boolean algebra 
 $\CLOP(\St(\bool{B}))$
 of clopen subset of $\St(\bool{B})$.
 Moreover a boolean algebra $\bool{B}$ is complete if and only if 
 $\CLOP(\St(\bool{B}))=\RO(\St(\bool{B}),\tau_{\bool{B}})$. Spaces $X$ satisfying the property that their
 regular open sets are closed are extremally (or extremely) disconnected.
 
 We refer the reader to \cite{GIVANTHALMOS} or \cite[Chapter 1]{viale-notesonforcing} (available at the
 following hyperlink:~\emph{\href{http://www.logicatorino.altervista.org/matteo_viale/dispenseTI2014.pdf}{Notes on Forcing}})
 for a detailed account of these matters.

 \subsection{Boolean valued models}

In a first order model, a formula can be interpreted as true or false. 
Given a complete boolean algebra $\bool{B}$,
$\bool{B}$-boolean valued models generalize Tarski semantics associating to each formula a value in $\bool{B}$, 
so that propositions 
are not only true and false anymore (that is, only associated to 
$1_{\bool{B}}$ and $0_{\bool{B}}$ respectively),
but take also other ``intermediate values'' of truth.
A complete account of the theory of these boolean
valued models can be found in \cite{RASIOWASIKORSKI1970}. 
We now recall some basic facts; an expanded version of the material of this section can be found 
in~\cite{VIAVAC15} (see also the following hyperlink:~\emph{\href{http://www.logicatorino.altervista.org/matteo_viale/thesis-vaccaro.pdf}{Tesi-Vaccaro}}) and in \cite[Chapter 3]{viale-notesonforcing}.
In order to avoid unnecessary technicalities, we define 
boolean valued semantics just for
relational first order languages (i.e. signatures with no function symobols).

%Since this book is a bit out of date, we recall below the basic facts we will need and we invite 
%the reader to consult \cite[Chapter 3]{vaccaro-tesi} for a detailed account on the material of this section.
%In order to avoid technicalities we won't need to bother with, we define 
%boolean valued semantics just for
%relational first order languages (i.e. signatures with no function symobols).

\begin{definition} \label{str2}
Given a complete boolean algebra $\bool{B}$ and a first order relational language 
\[ 
\mathcal{L}=\left\{R_i :i\in I \right\}\cup\bp{c_j:j\in J}
\]
a \emph{$\bool{B}$-boolean valued model} (or $\bool{B}$-valued model) 
$\mathcal{M}$ in the language $\mathcal{L}$ is a tuple 
\[ 
\langle M, =^\mathcal{M}, R_i^\mathcal{M} : i \in I, c_j^\mathcal{M} : j \in J\rangle
 \] 
where:
\begin{enumerate}
\item \label{nn1} $M$ is a non-empty set, called \emph{domain} of the 
$\bool{B}$-boolean valued model, whose elements are called \emph{$\bool{B}$-names};
\item $=^\mathcal{M}$ is the \emph{boolean value} of the equality: 
\begin{align*}
=^\mathcal{M}:M^2 &\to \bool{B} \\
(\tau,\sigma) &\mapsto \Qp{\tau=\sigma}^{\mathcal{M}}_{\bool{B}}
\end{align*}
\item The forcing relation $R_i^{\mathcal{M}}$ is the \emph{boolean interpretation}
of the $n$-ary relation symbol $R_i$:
\begin{align*}
R_i^{\mathcal{M}}:M^n &\to \bool{B} \\
(\tau_1,\dots,\tau_n) &\mapsto \Qp{R_i(\tau_1,\dots, \tau_n)}^{\mathcal{M}}_{\bool{B}}
\end{align*}
\item $c_j^{\mathcal{M}}\in M$ is the \emph{boolean interpretation}
of the constant symbol $c_j$.
\end{enumerate}

We require that the following conditions hold:
\begin{itemize}
\item for $\tau,\sigma,\chi\in M$,
\begin{enumerate}%[(i)]
\item %\label{i} 
$\Qp{\tau=\tau}^{\mathcal{M}}_{\bool{B}}=1_{\bool{B}}$;
\item 
$\Qp{\tau=\sigma}^{\mathcal{M}}_{\bool{B}}=\Qp{\sigma=\tau}^{\mathcal{M}}_{\bool{B}}$;
\item 
$\Qp{\tau=\sigma}^{\mathcal{M}}_{\bool{B}}
\wedge\Qp{\sigma=\chi}^{\mathcal{M}}_{\bool{B}} \le \Qp{\tau=\chi}^{\mathcal{M}}_{\bool{B}}$;
\end{enumerate}
\item for $R\in \mathcal{L}$ with arity $n$, and
$(\tau_1,\dots, \tau_n),(\sigma_1,\dots, \sigma_n) \in M^n$,
\begin{enumerate}%[(iv)]
\item  $ ( \bigwedge_{h \in \left\{1, \dots, n  \right\} }
  \Qp{\tau_h=\sigma_h}^{\mathcal{M}}_{\bool{B}}  ) \wedge 
  \Qp{R(\tau_1, \dots, \tau_n)}^{\mathcal{M}}_{\bool{B}}
    \le \Qp{ R(\sigma_1, \dots, \sigma_n)}^{\mathcal{M}}_{\bool{B}} $;
\end{enumerate}
\end{itemize}
%If no confusion can arise, we will confuse a predicate symbol with its interpretation.

Given a $\bool{B}$-model $\langle M,=^M\rangle$ for equality, a forcing relation $R$ on $M$ is a
map $R:M^n\to\bool{B}$ satisfying the above condition for boolean predicates.  
\end{definition}

The boolean valued semantics is defined as follows: 
\begin{definition}\label{BVS}
Let
\[ 
\langle M, =^\mathcal{M}, R_i^\mathcal{M} : i \in I, c_j^\mathcal{M} : j \in J\rangle
\] 
be a $\bool{B}$-valued model in a relational language 
\[ 
\mathcal{L}=\left\{R_i :i\in I \right\}\cup\bp{c_j:j\in J},
\]
$\phi$ a $\mathcal{L}$-formula
whose free variables are in $\bp{x_1,\dots,x_n}$,
and $\nu$ a
valuation of the free variables in $\mathcal{M}$ whose
domain contains $\bp{x_1,\dots,x_n}$.
Since $\mathcal{L}$ is a relational language, the terms of a formula are either free variables or constants,
let us define $\nu(c_j)=c^M_j$ for $c_j$ a constant of $\mathcal{L}$.
We denote with $\Qp{\phi}^{\mathcal{M},\nu}_{\bool{B}}$ the \emph{boolean value} of $\phi$ with the assignment
$\nu$.

%First, let $t$ be an $\mathcal{L}$-term and $\tau\in M$; we define recursively
%$\Qp{(t=\tau)(\nu)}^{\mathcal{M}}_{\bool{B}}\in \bool{B}$ as follows:
%\begin{itemize}
%\item if $t$ is a variable $x$, then 
%\[ 
%\Qp{(x=\tau)(\nu)}^{\mathcal{M}}_{\bool{B}}=\Qp{\nu(x)=\tau}^{\mathcal{M}}_{\bool{B}} \]
%\item if $t=f(t_1,\dots,t_n)$ where $t_i$ are terms and $f$ is an $n$-ary function symbol, then 
%\[
% \Qp{(f(t_1,\dots,t_n)=\tau)(\nu)}^{\mathcal{M}}_{\bool{B}}= \bigvee_{\sigma_1,\dots,\sigma_n\in M} 
% \cp{\bigwedge_{1\le i \le n} \Qp{(t_i=\sigma_i)(\nu)}^{\mathcal{M}}_{\bool{B}}}   
% \wedge \Qp{f(\sigma_1,\dots,\sigma_n)=\tau}^{\mathcal{M}}_{\bool{B}}
%  \]
%\end{itemize}
Given a formula $\phi$, we define recursively $\Qp{\phi}^{\mathcal{M},\nu}_{\bool{B}}$ as follows:
\begin{itemize}
\item for atomic formulae this is done letting
\[
\Qp{t=s}^{\mathcal{M},\nu}_{\bool{B}} = \Qp{\nu(t)=\nu(s)}^{\mathcal{M}}_{\bool{B}},
\]
and
\[ 
\Qp{R(t_1,\dots,t_n)}^{\mathcal{M},\nu}_{\bool{B}} =
\Qp{R(\nu(t_1),\dots,\nu(t_n))}^{\mathcal{M}}_{\bool{B}}
 \]
\item if $\phi\equiv \lnot \psi$, then 
\[ 
\Qp{\phi}^{\mathcal{M},\nu}_{\bool{B}}=\lnot \Qp{\psi}^{\mathcal{M},\nu}_{\bool{B}};
\]
\item if $\phi \equiv \psi \wedge \theta$, then
 \[
\Qp{\phi}^{\mathcal{M},\nu}_{\bool{B}}=\Qp{\psi}^{\mathcal{M},\nu}_{\bool{B}}
\wedge \Qp{\theta}^{\mathcal{M},\nu}_{\bool{B}}; 
\]
\item if $\phi\equiv \exists y \psi(y)$, then 
\[
 \Qp{\phi}^{\mathcal{M},\nu}_{\bool{B}}= \bigvee_{\tau\in M} \Qp{\psi(y/\tau)}^{\mathcal{M},\nu}_{\bool{B}};
 \]
\end{itemize}
If no confusion can arise, we omit the superscripts $\mathcal{M},\nu$ and the subscript $\bool{B}$, 
and we simply denote
the boolean value of a formula $\phi$ with parameters in $\mathcal{M}$ by $\Qp{\phi}$.
\end{definition}

%By definition, an isomorphism of boolean valued models preserves the boolean value of the 
%atomic formulas.
%Proceeding by induction on the complexity, one can get the result for any formula.
%
%\begin{proposition}
%Let $\mathcal{M}$ be a $\bool{B}$-valued model and $\mathcal{N}$ a $\mathsf{C}$-valued model in the same
%language $\mathcal{L}$. Assume $\langle i, \Phi \rangle$ is an isomorphism of boolean valued models.
%Then for any $\mathcal{L}$-formula $\phi(x_1,\dots,x_n)$, 
%and for every $(\tau_1,\sigma_1), \dots, (\tau_n,\sigma_n) \in \Phi$ we have that:
%\[
% i(\Qp{\phi(\tau_1,\dots,\tau_n)}^{\mathcal{M}}_{\bool{B}})=
% \Qp{\phi(\sigma_1,\dots,\sigma_n)}^{\mathcal{N}}_{\mathsf{C}}
%  \]
%\end{proposition}

With elementary arguments it is
possible prove the Soundness Theorem for boolean
valued models.

\begin{theorem}[Soundness Theorem] \label{soundness}
Assume $\mathcal{L}$ is a relational language and 
$\phi$ is a $\mathcal{L}$-formula which
is syntactically provable by
a $\mathcal{L}$-theory $T$. Assume each formula in $T$ has boolean
value at least $b\in \bool{B}$ in a $\bool{B}$-valued model $\mathcal{M}$ with valuation $\nu$.
Then $\Qp{\phi}^{\mathcal{M},\nu}_{\bool{B}}\geq b$ as well.
\end{theorem}
On the other hand the completeness theorem for the boolean valued semantics with respect to first order calculi 
is a triviality, given that $2$ is complete boolean algebra.

We get a standard  Tarski model from a $\bool{B}$-valued model by passing to a 
quotient by an ultrafilter $G\subseteq \bool{B}$. 
%This corresponds for spaces of type 
%$C^+(St(\bool{B}))$ to a specialization of the space to the
%ring of germs in $G$. In the general context it is defined as follows.

\begin{definition} \label{ultraquot}
Take $\bool{B}$ a complete boolean algebra, $\mathcal{M}$ a $\bool{B}$-valued model in the language 
$\mathcal{L}$, and $G$ an ultrafilter over $\bool{B}$.
Consider the following equivalence relation on $M$:
\[ 
\tau \equiv_G \sigma \Leftrightarrow \Qp{\tau=\sigma}\in G 
\]
The first order model
$\mathcal{M}/G =\langle M/G, R_i^{\mathcal{M}/G} : i \in I,
c_j^{\mathcal{M}/G} : j \in J \rangle$ is defined letting:
\begin{itemize}
\item For any  $n$-ary relation symbol $R$ in $\mathcal{L}$
\[
R^{\mathcal{M}/G}=\bp{([\tau_1]_G,\dots,[\tau_n]_G)\in (M/G)^n: \Qp{R(\tau_1,\dots,\tau_n)}\in G}.
\]
\item For any constant symbol $c$ in $\mathcal{L}$
\begin{align*}
c^{\mathcal{M}/G}=[c^{\mathcal{M}}]_G.
\end{align*}
\end{itemize}
\end{definition}

If we require $\mathcal{M}$ to satisfy a key additional condition, we get an easy
way to control the truth value of formulas in $\mathcal{M}/G$.

\begin{definition}
A $\bool{B}$-valued model $\mathcal{M}$ for the language $\mathcal{L}$ is \emph{full} if for every
$\mathcal{L}$-formula $\phi(x,\bar{y})$ and every $\bar{\tau} \in M^{\lvert \bar{y} \rvert}$
there is a $\sigma\in M$ such that
\[
 \Qp{\exists x \phi(x, \bar{\tau})}  =  \Qp{\phi(\sigma, \bar{\tau}) }
 \] 
\end{definition}

\begin{theorem}[\L o\'{s}'s Theorem for Boolean Valued Models] \label{them:LOSfbvm}
Assume $\mathcal{M}$ is a full $\bool{B}$-valued model for the relational language $\mathcal{L}$.
Then for every formula $\phi(x_1, \dots, x_n)$ in $\mathcal{L}$ and $(\tau_1,\dots,\tau_n)\in M^n$: 
\begin{enumerate}%[(i)]
\item 
For all ultrafilters $G$ over $\bool{B}$
\[
\mathcal{M}/G\models \phi([\tau_1]_G,\dots,[\tau_n]_G)\text{ if and only if }
\Qp{\phi(\tau_1,\dots,\tau_n)}\in G.
\]
%\smallskip

\item
For all $a\in \bool{B}$ the following are equivalent:
\begin{enumerate}%[(a)]
\item
$\Qp{\phi(f_1,\dots,f_n)}\geq a$,
\item
$\mathcal{M}/G\models \phi([\tau_1]_G,\dots,[\tau_n]_G)$ for all $G\in N_a$,
\item
$\mathcal{M}/G\models \phi([\tau_1]_G,\dots,[\tau_n]_G)$ for densely many $G\in N_a$.
\end{enumerate}
\end{enumerate}
\end{theorem}

A key observation to relate standard ultraproducts to boolean valued models is the following:
\begin{fact}
Let $(M_x:x\in X)$ be a family of Tarski-models in the first order relational language $\mathcal{L}$.
Then $N=\prod_{x\in X}M_x$ is a full $\pow{X}$-model letting for 
each $n$-ary relation symbol $R\in \mathcal{L}$, 
$\Qp{R(f_1,\dots,f_n)}_{\pow{X}}=\bp{x\in X: M_x\models R(f_1(x),\dots,f_n(x))}$.
\end{fact}
Let $G$ be any non-principal ultrafilter on $X$. Then, using the notation of the previous fact,
$N/G$ is the familiar ultraproduct of the family $(M_x:x\in X)$ by $G$, and
the usual \L{}o\'s Theorem for ultraproducts of Tarski models 
is the specialization to the case of the full $\pow{X}$-valued model $N$
of Theorem~\ref{them:LOSfbvm}. 
Notice that in this special case, if the ultraproduct is an ultrapower of a model $M$,
 the embedding
$a\mapsto [c_a]_G$ (where $c_a(x)=a$ for all $x\in X$ and $a\in M$) is elementary. 

%\subsection{Standard ultraproducts are quotients of full boolean valued models}

%Notice also that $\bp{G_x}$ is an isolated point in $St(\power{X})$ for all $x\in X$.
%In particular we get that a function $f:X\to M$, where $M$ is endowed with the discrete topology,

\subsection{Boolean ultrapowers of compact Hausdorff spaces and Shoenfield's absoluteness}

Take $X$ a set with the discrete topology, and for any $a\in X$, let
$G_a\in\St(\pow{X})$ denote the principal ultrafilter given by supersets of $\bp{a}$.
The map $a\mapsto G_a$ embeds $X$ as an open, dense, discrete subspace of $\St(\pow{X})$.
In particular for any topological space $(Y,\tau)$, any function $f:X\to Y$ is continuous (since $X$ has the discrete
topology); moreover in case $Y$ is compact Hausdorff it induces a unique continuous 
$\bar{f}:\St(\pow{X})\to Y$ mapping $G\in \St(\pow{X})$ to the unique point in $Y$ which is in the intersection of
$\bp{\Cl{A}: A\in \tau, f^{-1}[A]\in G}$ (we are in the special situation in which $\St(\pow{X})$ is also the 
Stone-\v{C}ech compactification of $X$).

This gives that for any compact Hausdorff space $(Y,\tau)$, the space $C(X,Y)=Y^X$ of (continuous) functions
from $X$ to $Y$ is canonically isomorphic to the space $C(\St(\pow{X}),Y)$ of continuous functions from
$\St(\pow{X})$ to $Y$. 

What if we replace $\pow{X}$ with an arbitrary (complete) boolean algebra?
In view of the above remarks, it 
is reasonable to regard $C(\St(\bool{B}),Y)$ is the $\bool{B}$-ultrapower of $Y$ for any compact 
Hausdorff space $Y$, since this is exactly what occurs for the case $\bool{B}=\pow{X}$.

Let us examine closely this situation in the case $Y=2^\omega$ with the product topology.
This will unfold the relation existing 
between the notion of boolean ultrapowers of $2^\omega$  and Shoenfield's absoluteness. 
%for
%$\Sigma^1_2$-properties a natural analogue of the second part of \L o\'s theorem stating that a first order
%structure embeds as an elementary substructure of its ultrapowers.

Let us fix $\bool{B}$ arbitrary (complete) boolean algebra, and
set $M=C(St(\bool{B}),2^\omega)$. 
Fix also $R$ a Borel relation on $(2^\omega)^n$.
The continuity of  an $n$-tuple $f_1,\dots,f_n\in M$ implies that the set
%\Rnote{I suppressed $R^M$ that is defined right after}
\[
\{G: R(f_1(G)\dots,f_n(G))\}=(f_1\times \dots\times f_n)^{-1}[R]
\]
has the Baire property in $\St(\bool{B})$ (i.e. it has meager symmetric difference with
a unique regular open set --- see~\cite[Lemma 11.15, Def. 32.21]{JECH}), where 
$f_1\times \dots\times f_n(G)=(f_1(G),\dots,f_n(G))$.
So we can define 
%\Rnote{suggestion: use Reg everywhere?}
%\Rnote{I used the notation $M$ everywhere.}
\begin{align*}
R^M: & M^n\to\bool{B} \\
& (f_1,\dots,f_n)=\Reg{\bp{G: R(f_1(G),\dots, f_n(G)}}.
\end{align*}
Also, since the diagonal is closed in $(2^\omega)^2$,
\[
=^M(f,g)=\Reg{\bp{G: f(G)= g(G)}}
\] 
is well defined.

It is not hard to check that, for any Borel relation $R$ on $(2^\omega)^n$, 
the structure ${(M,=^M,R^M)}$ is a full $\bool{B}$-valued extension of $(2^\omega,=,R)$, 
where $2^\omega$ is copied inside 
$M$ as the set of constant functions. 
It is also not hard to check that
whenever $G$ is an ultrafilter on $St(\bool{B})$,
the map $i_G: 2^\omega\to M/G$ given 
by $x\mapsto [c_x]_G$ (the constant function with value $x$)
defines an injective morphism of the $2$-valued structure
$(2^\omega,R)$ into the $2$-valued structure $(M/G,R^M/G)$.
Nonetheless it is not clear whether this morphism is an elementary map or not. This is the case for
$\bool{B}=\pow{X}$, since in this case we are analyzing the standard embedding
of the first order structure  $(2^\omega,R)$ in its ultrapowers induced by ultrafilters on $\pow{X}$. What
are the properties of this map if $\bool{B}$ is some other complete boolean algebra?

We can relate the degree of elementarity of the map $i_G$
with Shoenfield's absoluteness for $\Sigma^1_2$-properties. 
This can be done if one is willing to accept as a black-box 
the identification of the $\bool{B}$-valued model $C(St(\bool{B}),2^\omega)$ with the $\bool{B}$-valued model given by
the family of $\bool{B}$-names for elements of $2^\omega$ in
$V^{\bool{B}}$ (which is the canonical $\bool{B}$-valued model for set theory); we will expand further 
on this identification in the next section. 
Modulo this identity, 
Shoenfield's absoluteness
can be recast as a statement about boolean valued models. 
We choose to name Cohen's absoluteness the following statement, which gives (as we will see) an equivalent
reformulation of Shoenfield's absoluteness. Its proof
(as we will see in the next section)
ultimately relies on Cohen's
forcing theorem, hence the name.
\begin{theorem}[Cohen's absoluteness]\label{thm:cohabs}
Assume $\bool{B}$ is a complete boolean algebra and $R\subseteq (2^\omega)^n$ is a Borel relation.
Let $M=C(St(\bool{B}),2^\omega)$  and $G\in \St(\bool{B})$.
Then
\[
(2^\omega,=,R)\prec_{\Sigma_2}(M/G,=^M/G,R^M/G).
\]
\end{theorem}
%\Rnote{Why Cohen? why not someone else?}

\section{Getting Cohen's absoluteness from Baire's category Theorem} \label{sec:CABCT}
Let us now show how Theorem \ref{thm:cohabs} is once again a consequence of forcing axioms.
To do so, we delve deeper into set theoretic techniques and assume 
the reader has some acquaintance with the forcing method.
We give below a brief review sufficient for our aims.

\subsection{Forcing}
Let $V$ denote the standard universe of sets and $\ZFC$ the standard first order axiomatization of set theory by the 
Zermelo-Frankel axioms.
For any complete boolean algebra $\bool{B}\in V$ let 
\[
V^{\bool{B}}=\bp{f:V^{\bool{B}}\to \bool{B}:\, f\in V\text{ is a function}}
\]
be the class of $\bool{B}$-names
with boolean relations $\in^{\bool{B}},\subseteq^{\bool{B}},=^{\bool{B}}:(V^{\bool{B}})^2\to\bool{B}$
given by:
\begin{enumerate}
\item 
%$\apri \tau \in \sigma \chiudi =$
\[ 
\in^{\bool{B}}(\tau,\sigma)=
\Qp{\tau \in \sigma} =
\bigvee_{\tau_0 \in \dom(\sigma) } ( \Qp{\tau = \tau_0} \land \sigma(\tau_0)).\] 
\item 
%$\apri \tau \subseteq \sigma \chiudi = $
\[ 
\subseteq^{\bool{B}}(\tau,\sigma)=
%\Qp{ \tau \subseteq \sigma }= \bigwedge_{\sigma_0 \in \dom(\tau)} (\tau(\sigma_0) \rightarrow 
%\Qp{\sigma_0 \in \sigma} )= 
\bigwedge_{\sigma_0 \in \dom(\tau)}(\neg \tau(\sigma_0) \lor\Qp{\sigma_0 \in \sigma}).
\]
\item 
%$
\[
=^{\bool{B}}(\tau,\sigma)=\Qp{\tau = \sigma}= \Qp{\tau \subseteq \sigma}\land\Qp{\sigma \subseteq \tau }.
\]
%$.
\end{enumerate}

\begin{theorem}[Cohen's forcing theorem $I$]
$(V^{\bool{B}},\in^{\bool{B}},=^{\bool{B}})$ is a full boolean valued model which assigns the boolean
value $1_{\bool{B}}$ to
all axioms $\phi\in \ZFC$.
\end{theorem}

$V$ is copied inside $V^{\bool{B}}$ as the family of $\bool{B}$-names
$\check{a}=\bp{\ap{\check{b},1_{\bool{B}}}:b\in a}$ and has the property that for all
$\Sigma_0$-formulae (i.e with quantifiers bounded to range over sets)
$\phi(x_0,\dots,x_n)$ and $a_0,\dots,a_n\in V$
\[
\Qp{\phi(\check{a}_0,\dots,\check{a}_n)}=1_{\bool{B}}
\mbox{ if and only if }
V\models\phi(a_0,\dots,a_n).\]

This procedure can be formalized in any first order model $(M,E,=)$ of $\ZFC$ for any
$\bool{B}\in M$ such that $(M,E,=)$ models that $\bool{B}$ is a complete boolean algebra.

Two ingredients are still missing to prove Cohen's absoluteness (Theorem~\ref{thm:cohabs})
from Baire's category theorem: the notion of an $M$-generic filter and the duality between 
$C(\St(\bool{B}),2^\omega)$ and the $\bool{B}$-names in $V^{\bool{B}}$ for elements of $2^\omega$.
We first deal with the duality.

\subsection{$C(\St(\bool{B}),2^\omega)$ is the family of $\bool{B}$-names for elements of $2^\omega$}\label{subsec:cstbisovb}

\begin{definition}
Let $\bool{B}$ be a complete boolean algebra.
Let $\sigma\in V^{\bool{B}}$ be a $\bool{B}$-name such
that $\Qp{\sigma: \check{\omega} \to \check{2}}_{\bool{B}}=1_{\bool{B}}$. 
We define $f_\sigma: \St(\bool{B}) \to 2^\omega$ by
\[
	f_\sigma(G)(n)=i \iff \Qp{\sigma(\check{n})=\check{i}}\in G.
\]
Conversely assume $g:\St(\bool{B}) \to 2^\omega$ is a continuous function, then define
\[
	\tau_g=\{\langle\check{(n, i)}, \{G: g(G)(n)=i\}\rangle: n\in\omega, i < 2\} \in V^{\bool{B}}.
\]
\end{definition}

%Note that we consider $2^\omega$ as a topological space with the product topology. 
Observe indeed that 
\[
	\{G\in \St(\bool{B}): g(G)(n)=i\}= g^{-1}[N_{n,i}],
\]
where $N_{n,i}=\{f \in 2^\omega : f(n) = i\}$. Since 
$g$ is continuous, $g^{-1}[N_{n,i}]$ is clopen and so it is an element of $\bool{B}$.

We can prove the following duality:

\begin{proposition}\label{proposition: tauf-ftau}
Assume that $\Qp{\sigma: \check{\omega} \to \check{2}}_{\bool{B}}=1_{\bool{B}}$ 
and $g: \St(\bool{B}) \to 2^\omega$ is continuous. Then
\begin{enumerate}
\item $\tau_g\in V^{\bool{B}}$;
\item $f_\sigma: \St(\bool{B}) \to 2^\omega$ is continuous;\label{item: fcont}
\item $\Qp{\tau_{f_\sigma}=\sigma}_{\bool{B}}=1_{\bool{B}}$;
\item $f_{\tau_g}=g$.
\end{enumerate}
In particular letting
\[
(2^{\omega})^{\bool{B}}=\bp{\sigma\in V^{\bool{B}}: \Qp{\sigma: \check{\omega} \to \check{2}}_{\bool{B}}=1_{\bool{B}}},
\]
the $2$-valued models $((2^{\omega})^{\bool{B}}/G,=^{\bool{B}}/G)$ and 
$(C(\St(\bool{B}),2^\omega)/G,=^{\St(\bool{B})}/G)$ are isomorphic for all $G\in \St(\bool{B})$ via the map
$[g]_G\mapsto[\tau_g]_G$.
\end{proposition}

This is just part of the duality, as the duality can lift the isomorphism also to all 
$\bool{B}$-Baire relations on $2^\omega$, among which are all Borel relations.
Recall that for any given topological space $(X,\tau)$ a subset $Y$ of $X$ is meager for $\tau$
if $Y$ is contained in the countable union of closed nowhere dense (i.e. with complement dense open)
subsets of $X$. $Y$ has the Baire property if $Y\Delta A$ is meager for some unique regular open set $A\in\tau$.

\begin{definition}
$R\subseteq (2^\omega)^n$ 
is a $\bool{B}$-Baire subset of $(2^\omega)^n$ if 
for all continuous functions $f_1,\dots,f_n:\St(\bool{B})\to 2^\omega$ we have that
\[
(f_1\times\dots\times f_n)^{-1}[A]=\bp{G: f_1\times\dots\times f_n(G)\in A}
\]
has the Baire property in $\St(\bool{B})$.

$R\subseteq (2^\omega)^n$ is universally Baire if it is $\bool{B}$-Baire for all
complete boolean algebras $\bool{B}$.
\end{definition}
It can be shown in $\ZFC$ 
that Borel (and even analytic) subsets of $(2^\omega)^n$ are universally Baire (see~\cite[Def. 32.21]{JECH}).

An important result of Feng, Magidor, and Woodin \cite{FENMAGWOO} can be restated as follows:
\begin{theorem}\label{thm:dual2omegaBnames}
$R\subseteq (2^\omega)^n$ is $\bool{B}$-Baire if and only if there exist
$\dot{R}^{\bool{B}}\in V^{\bool{B}}$ such that 
\[
\Qp{\dot{R}^{\bool{B}}\subseteq \check{(2^\omega)^n}}=1_{\bool{B}},
\]
and
for all $\tau_1,\dots,\tau_n\in (2^\omega)^{\bool{B}}$
\[
\Reg{\bp{G: R(f_{\tau_1}(G),\dots,f_{\tau_n}(G))}}=
\Qp{(\tau_1,\dots,\tau_n)\in\dot{R}^{\bool{B}}}.
\]
\end{theorem}

In particular an easy corollary is the following:
\begin{theorem}
Let $R\subseteq (2^\omega)^n$ be a $\bool{B}$-baire relation.
Then the map $[f]_G\mapsto [\tau_f]_G$ implements an isomorphism between
\[
\ap{C(\St(\bool{B})/G,R^{\St(\bool{B})}/G}\cong \ap{(2^{\omega})^{\bool{B}}/G,\dot{R}^{\bool{B}}/G}
\]
for any $G\in\St(\bool{B})$.
\end{theorem}
These results can be suitably generalized to arbitrary Polish spaces. We refer the reader to~\cite{VIAVAC15} 
and~\cite{vaccaro-tesi}. \cite{VIASCHANUEL} gives an application of this result to tackle a problem in number theory
related to Schanuel's conjecture.

%\begin{aRnote}{On references}
%I personally feel uneasy refering the reader to a master thesis, especially since it is not published. It may be nitpicking, but a reasonable guideline could be:
%
%when there are no clear reference, saying it. Then \emph{refering} the reader is quite strong, it implies that Vaccaro's thesis is a reference. I would be careful and say simply "some details can be found in"
%\end{aRnote}

\subsection{$M$-generic filters and Cohen's absoluteness}

\begin{definition}
Let $(P,\leq)$ be a partial order and $M$ be a set.
A subset $G$ of $P$ is $M$-generic if $G\cap D$ is non-empty 
for all $D\in M$ predense subset of $P$.
\end{definition}

By $\mathsf{BCT}_1$ every countable set $M$ admits $M$-generic filters for all partial orders $P$.

\begin{theorem}[Cohen's forcing theorem II]
Assume $(N,\in)$ is a transitive model of $\ZFC$,
$\bool{B}\in N$ is a complete boolean algebra in $N$, and
$G\in\St(\bool{B})$ is an $N$-generic filter for $\bool{B}^+$.

Let
\begin{align*}
\val_G:& N^{\bool{B}}\to V\\
& \sigma\mapsto\sigma_G=\bp{\tau_G: \exists b\in G\, \ap{\tau,b}\in\sigma},
\end{align*}
and $N[G]=\val_G[N^{\bool{B}}]$.

Then $N[G]$ is transitive, the map $[\sigma]_G\mapsto\sigma_G$ is the Mostowski collapse
of the Tarski models $\ap{N^{\bool{B}}/G,\in^{\bool{B}}/G}$ and induces an isomorphism of
this model with the model $\ap{N[G],\in}$.

In particular for all formulae $\phi(x_1,\dots,x_n)$ and $\tau_1\dots,\tau_n\in N^{\bool{B}}$
\[
\ap{N[G],\in}\models\phi((\tau_1)_G,\dots,(\tau_n)_G)
\]
if and only if $\Qp{\phi(\tau_1,\dots,\tau_n)}\in G$.
\end{theorem}
Recall that:
\begin{itemize}
\item
For any infinite cardinal $\lambda$, $H_\lambda$ is the set of all sets $a\in V$ such that
$|\trcl(a)|<\lambda$ (where $\trcl(a)$ is the transitive closure of the set $a$).
\item
If $\kappa$ is a strongly inaccessible cardinal (i.e. regular and strong limit),
$H_\kappa$ is a transitive model of $\ZFC$.
\item
A property
$R\subseteq (2^\omega)^n$ is $\Sigma^1_2$, if it is of the form
\[
R=\bp{(a_1,\dots,a_n)\in (2^\omega)^n: \exists y\in 2^\omega\,\forall x\in 2^\omega \,S(x,y,a_1,\dots,a_n)}
\]
with $S\subseteq (2^\omega)^{n+2}$ a Borel relation.
\item
If $\phi(x_0,\dots,x_n)$ is a $\Sigma_0$-formula and $M\subseteq N$ are transitive sets or classes,
then for all $a_0,\dots,a_n\in M$
\[
M\models \phi(a_0,\dots,a_n)\text{ if and only if }N\models \phi(a_0,\dots,a_n).
\]
\end{itemize}

%\Rnote*{expand a little on this?}{
Observe that %$H_{\omega_1}\prec_{\Sigma_1} V$ (this is Levy's absolutness~\cite[]{JECH} and that for 
for any theory $T\supseteq\ZFC$ there is a recursive translation of
 $\Sigma^1_2$-properties (provably $\Sigma^1_2$ over $T$)
 into $\Sigma_1$-properties over $H_{\omega_1}$ (provably $\Sigma_1$ over the same theory $T$)
\cite[Lemma 25.25]{JECH}. 
%}

\begin{lemma}%[Cohen's Absoluteness]
\label{lem:CohAbs}
Assume $\phi(x,r)$ is a $\Sigma_0$-formula in the parameter $\vec{r}\in (2^\omega)^n$.
Then the following are equivalent:
\begin{enumerate}
\item
$H_{\omega_1}\models\exists x\phi(x,r)$.
\item
For all complete boolean algebra $\bool{B}$
$\Qp{\exists x\phi(x,r)}=1_{\bool{B}}$.
\item
%$T\vdash \exists x\phi(x,r)$\emph{ is $\Omega$-consistent\footnote{I.e.
There is a complete boolean algebra $\bool{B}$ such that 
$\Qp{\exists x\phi(x,r)}>0_{\bool{B}}$.
\end{enumerate}
\end{lemma}

Summing up we get: a $\Sigma^1_2$-statement holds in $V$
iff
the  corresponding $\Sigma_1$-statement over $H_{\omega_1}$ holds in some
model of the form $V^{\bool{B}}/G$. 

Combining the above Lemma with %\Rnote*{Which ones? or is it the previous subsection?}{
Proposition~\ref{proposition: tauf-ftau},
%} 
we can easily infer the proof
of Theorem~\ref{thm:cohabs}.

%This shows that already in $\ZFC$ forcing is an extremely
%powerful tool to prove theorems. 
%Lemma~\ref{Lem:CohAbs} complements 
%Shoenfield's absoluteness theorem~\cite[Theorem 25.20]{JECH}
%and gives another powerful argument to prove the validity of some 
%$\Sigma^1_2$-property by means of an absoluteness argument.
%with  stating that
%the truth value of a $\Sigma^1_2$-property 
%is the same in all transitive models $M$ of $\ZFC$ containing
%$\omega_1$~\cite[Theorem 25.20]{JECH}. 
%These two results are very similar in nature but the first one is more constructive. 
%For example a proof that a $\Sigma^1_2$-property holds in $L$
%does not yield automatically
%that this property is provable in $\ZFC$ but just that it holds in all uncountable transitive models of
%$\ZFC$ containing $\omega_1$; 
%yet this property could fail in some non-transitive model of $\ZFC$ or in some 
%transitive model of $\ZFC$ whose ordinals have order type less than $\omega_1$.

%We briefly sketch why Lemma~\ref{Lem:CohAbs}
%holds since this will outline many of the ideas we are heading for:
\begin%[Proof of Lemma~\ref{Lem:CohAbs}]
{proof}
We shall actually prove the following slightly
stronger formulation
%\footnote{In the statement below we do not require that
%the existence of a 
%partial order witnessing the $\Omega$-consistency of
%$\exists x\phi(x,r)$ in $V$ is provable in $T$.} 
of the non-trivial direction in the 
three equivalences above:
\begin{quote}
$H_{\omega_1}\models\exists x\phi(x,r)$ if 
$\Qp{\exists x\phi(x,r)}>0_{\bool{B}}$ for some complete boolean algebra $\bool{B}\in V$. 
\end{quote}
To simplify 
the exposition we prove this statement under the further assumption that 
that there exists an inaccessible cardinal $\kappa>\bool{B}$.
With greater care for details the large cardinal assumption can be removed.
So assume $\phi(x,\vec{y})$ is a $\Sigma_0$-formula and
$\Qp{\exists x\phi(x,\check{\vec{r}})}>0_{\bool{B}}$ for some complete boolean algebra $\bool{B}\in V$
with parameters $\vec{r}\in (2^\omega)^n$.
Pick a model $M\in V$ such that 
$M\prec (H_{\kappa})^V$, $M$ is countable in $V$,  and
$\bool{B},\vec{r}\in M$. 
Let $\pi_M:M\to N$ be its transitive collapse (i.e. $\pi_M(a)=\pi_M[a\cap M]$ for all $a\in M$)
and $\bool{Q}=\pi_M(\bool{B})$. Notice also that 
$\pi_M(\vec{r})=\vec{r}$: since $\omega\in M$ is a definable ordinal \emph{contained} in $M$, 
$\pi_M(\omega)=\pi_M[\omega]=\omega$;
consequently, $\pi_M$ fixes also all the elements in $2^\omega\cap M$.

Since $\pi_M$ is an isomorphism of $M$ with $N$,
\[
N\models\ZFC\wedge(b=\Qp{\exists x\phi(x,\check{\vec{r}})}>0_{\bool{Q}}).
\] 
Now let $G\in V$ be $N$-generic for $\bool{Q}$ with $b\in G$ ($G$ exists since $N$ is countable);
then by Cohen's theorem of forcing applied in $V$ to $N$, 
we have that $N[G]\models\exists x\phi(x,\vec{r})$.
So we can pick $a\in N[G]$ such that $N[G]\models\phi(a,\vec{r})$. 
Since $N,G\in (H_{\omega_1})^V$, we have that
$V$ models that $N[G]\in H_{\omega_1}^V$ and thus $V$ models that 
$a$ as well belongs to $H_{\omega_1}^V$.
Since $\phi(x,\vec{y})$ is a $\Sigma_0$-formula, 
$V$ models that $\phi(a,\vec{r})$ is absolute between the
transitive sets $N[G]\subset H_{\omega_1}$
to which $a,\vec{r}$ belong. In particular
$a$ witnesses in $V$ that $H_{\omega_1}^V\models\exists x\phi(x,\vec{r})$. 
\end{proof}

\section{Maximal forcing axioms} \label{sec:MMWGA}

Guided by all the previous results we want to formulate maximal forcing axioms.
We pursue two directions:
\begin{enumerate}
\item
A direction shaped by topological considerations:
we have seen that $\FA_{\aleph_0}(P)$ holds for any partial order $P$, and that
$\mathsf{AC}$ is equivalent to the satisfaction of $\FA_{\lambda}(P)$ for all regular $\lambda$ and all
$<\lambda$-closed posets $P$.

We want to isolate the largest possible class of partial orders $\Gamma_\lambda$ for which 
$\FA_{\lambda}(P)$ holds for all $P\in \Gamma_\lambda$. The case $\lambda=\aleph_0$ is 
handled by Baire's category theorem, which shows that $\Gamma_{\aleph_0}$ is the class of all posets.
We will outline how the case $\lambda=\aleph_1$ is settled by the work of Foreman, Magidor, and Shelah \cite{foreman_magidor_shelah} and leads to
Martin's maximum. On the other hand, the case $\lambda>\aleph_1$ is wide open and until recently only 
partial results have been obtained. New techniques to handle the case 
$\lambda=\aleph_2$ are being developed 
(notably by Neeman, and also by Asper\`o, Cox, Krueger,
Mota, Velickovic, see among others ~\cite{krueger:forcingclubaleph2,krueger:mota:stronglyproper,neeman:forcingaleph2}),
however the full import of their possible applications is not clear yet.
\item
A direction shaped by model-theoretic considerations: Baire's category theorem implies
that the natural embedding of $2^\omega$ into $C(\St(\bool{B}),2^\omega)/G$ is $\Sigma_2$-elementary, whenever
$2^\omega$ is endowed with $\bool{B}$-baire predicates (among which are all the Borel predicates).
We want to reinforce this theorem in two directions: 
\begin{itemize}
\item[(A)]
We want to be able to infer that (at least for Borel predicates) the natural embedding of 
$2^\omega$ into $C(\St(\bool{B}),2^\omega)/G$
yields a full elementary embedding of $2^\omega$ into $C(\St(\bool{B}),2^\omega)/G$.
\item[(B)]
We want to be able to define boolean ultrapowers $M^{\bool{B}}$ also for 
other first order structures $M$ than $2^\omega$ and be able to
infer that the natural embedding of $M$ into $M^{\bool{B}}/G$ is elementary for these boolean ultrapowers.
\end{itemize}
\end{enumerate}
Both directions (the topological and the model-theoretic)
converge towards the isolation of certain natural forcing axioms.   
Moreover for each cardinal $\lambda$, the relevant stuctures for which we can define a natural notion of
boolean ultrapower are either the structure
$H_{\lambda^+}$, or the Chang model $L(\Ord^{\lambda})$. 

We believe that we have now a satisfactory 
understanding  of the maximal forcing axioms one can get following both directions for the cases
$\lambda=\aleph_0,\aleph_1$. The main open question remaining how to isolate (if at all possible) the maximal
forcing axioms for $\lambda>\aleph_1$. 

\subsection{Woodin's generic absoluteness for $H_{\omega_1}$ and $L(\Ord^\omega)$}

%Let us first approach our quest for maximal forcing axioms following the model-theoretic direction.
%In what follows we are inspired by
%Woodin's work in $\Omega$-logic.
%The basic observation is that the 
%working tools of a set theorist are either 
%first order calculus, by which he/she can justify his/her proofs over $\ZFC$, or forcing,
%by which he/she can obtain his/her independence results over $\ZFC$.
%However it appears that there is still a gap between what we can achieve by ordinary proofs in 
%some axiom system which extends $\ZFC$ and the independence 
%results that we can obtain over this theory by means of forcing. 
%More specifically to close the gap it appears that
%we are lacking two desirable feature we would like to have for a ``complete'' first order theory $T$
%that axiomatizes set theory with respect to its semantics given by the class of boolean 
%valued models of $T$:

We start with the model-theoretic direction, following Woodin's work in $\Omega$-logic.
Observe that a set theorist works either with 
first order calculus to justify some proofs over $\ZFC$, or with forcing
to obtain independence results over $\ZFC$.
However, in axiom systems extending $\ZFC$ there seems to be a gap between what we can achieve by ordinary proofs and the independence 
results that we can obtain by means of forcing. 
To close this gap
we would like two desirable features of a ``complete'' first order theory $T$
that contains $\ZFC$, specifically with respect to the semantics given by the class of boolean 
valued models of $T$:
\begin{itemize}
\item
$T$ is complete with respect to its intended semantics, i.e for all statements $\phi$
only one among $T+\phi$ and $T+\neg\phi$ is forceable.
\item
Forceability over $T$ should correspond to a  notion of derivability with respect to some proof system, 
for instance derivability with respect to a standard first order calculus for $T$.
\end{itemize}

Both statements appear to be rather bold and have to be handled with care:
Consider for example the statement 
$|\omega|=|\omega_1|$ in a theory
$T$ extending $\ZFC$ with the statements \emph{$\omega$ is the first infinite cardinal} and 
\emph{$\omega_1$ is the first uncountable cardinal}.
Then clearly
$T$ proves $|\omega|\neq|\omega_1|$, while if one forces with $\Coll(\omega,\omega_1)$
one produces a model of set theory where this equality holds (however the formula
\emph{$\omega_1$ is the first uncountable cardinal} is now false in this model).

At first glance, this suggests that as we expand the language for $T$, forcing starts to act randomly
on the formulae of $T$, switching the truth value of its formulae with parameters in ways which
it does not seem simple to describe. 
However the above difficulties arise essentially from our lack of attention to define 
the type of formulae for which we aim to have the completeness of $T$ with respect to 
forceability. 
We can show that when the formulae are limited to talking only about a suitable initial segment
of the set theoretic universe (i.e. $H_{\omega_1}$ or $L(\Ord^\omega)$), 
and we consider only forcings that preserve the intended
meaning of the parameters by which we enriched the language of $T$ (i.e. parameters in $H_{\omega_1}$), 
this random
behaviour of forcing does not show up anymore.

We take a platonist's stance towards set theory; 
thus we have one canonical model $V$ of $\ZFC$, the truths of which we try to
uncover.
%To do this we may allow ourselves to use all model theoretic techniques that produce new models of
%the truths of $\mathsf{Th}(V)$ on which we are confident, which (if we are platonists)
%certainly include $\ZFC$ and all the axioms of large cardinals.
To do this, we may use model theoretic techniques that produce new models of
the part of $\mathsf{Th}(V)$ about which we are confident. This certainly includes $\ZFC$, and (for most platonists)
all the large cardinal axioms.

We may start our quest to uncover the truth in $V$ by first settling the theory of
$H_{\omega_1}^V$ (the hereditarily countable sets), then the theory of $H_{\omega_2}^V$
(the sets of hereditarily cardinality $\aleph_1$) and so on and so forth, thus covering step by step
all infinite cardinals.
To proceed we need some definitions:

 \begin{definition}
Given a theory $T\supseteq \ZFC$ and a family 
$\Gamma$ of partial orders definable in $T$, we say that
$\phi$ is $\Gamma$-consistent for $T$ if 
$T$ proves that there exists a complete boolean algebra $\bool{B}\in\Gamma$ 
such that $\Qp{\phi}_{\bool{B}}>0_{\bool{B}}$.

Given a model $V$ of $\ZFC$ we say that $V$ models that $\phi$ is
$\Gamma$-consistent if $\phi$ is $\Gamma$-consistent for $\mathsf{Th}(V)$.

\end{definition}

%\begin{definition}
%Given a partial order $P$ and a cardinal $\lambda$,
%$\FA_\lambda(P)$ holds if for all $\{D_\alpha:\alpha<\lambda\}$ family of dense subsets of
%$P$, there is $G\subset P$ filter which has non empty intersection with all
%$D_\alpha$.
%\end{definition}

\begin{definition}
Let 
\[
T\supseteq\ZFC+\{\lambda \text{ is an infinite cardinal}\}
\]
$\Omega_\lambda$ is the definable (in $T$) class of partial orders $P$ which satisfy
$\FA_\lambda(P)$.
\end{definition}
In particular Baire's category theorem amounts to saying that
$\Omega_{\aleph_0}$ is the class of all partial orders (denoted by Woodin
as the class $\Omega$).
The following is a careful reformulation of Lemma~\ref{lem:CohAbs} which 
does not require any ontological commitments about $V$.
\begin{lemma}[Cohen's Absoluteness Lemma]
Assume $T\supseteq\ZFC+\{p\subseteq\omega\}$ and
$\phi(x,p)$ is a $\Sigma_0$-formula.
Then the following are equivalent:
\begin{itemize}
\item
$T\vdash \exists x\phi(x,p)$,
\item
$\exists x\phi(x,p)$ is $\Omega$-consistent for $T$.
\end{itemize}
\end{lemma}
This shows that 
for $\Sigma_1$-formulae with real parameters
the desired overlap between the ordinary 
notion of provability and the semantic notion of forceability is provable in $\ZFC$.
Now it is natural to asking if we can expand the above in at least two directions:
\begin{enumerate}
\item
Increase the complexity of the formula,
\item
Expand the language allowing parameters also for other infinite cardinals.
\end{enumerate}
The second direction will be pursued in the next subsection.
Concerning the first direction, 
the extent by which we can increase the complexity of the formula requires once again
some attention to the semantical interpretation of its parameters and its quantifiers.
We have already observed that the formula $|\omega|=|\omega_1|$
is inconsistent but $\Omega$-consistent
in a language with parameters for $\omega$ and $\omega_1$.
One of 
Woodin's main achievements\footnote{We follow Larson's presentation as in~\cite{LARSON}.} 
in $\Omega$-logic shows that if we restrict the semantic interpretation of
$\phi$ to range over the structure $L([\Ord]^{\aleph_0})$ and 
we assume large cardinal axioms, we can get a full correctness and completeness 
result\footnote{The large cardinal assumptions on $T$ of the present formulation can be significantly reduced. 
See~\cite[Corollary 3.1.7]{LARSON}.}~\cite[Corollary 3.1.7]{LARSON}:

\begin{theorem}[Woodin]
Assume $T$ is a theory extending
\[
\ZFC+\{p\subset\omega\}+\text{there are class many supercompact cardinals},
\] 
$\phi(x,y)$ is any formula in free variables $x,y$, and $A\subseteq (2^{\omega})^n$ is universally Baire.
Then the following are equivalent (where $\dot{A}^{\bool{B}}$ is the $\bool{B}$-name given by the lifting of $A$ to $V^{\bool{B}}$ given by Theorem~\ref{thm:dual2omegaBnames} ):
\begin{itemize}
\item
$T\vdash [L([\Ord]^{\aleph_0},A)\models\phi(p,A)]$,
\item
$T\vdash \exists\bool{B}\Qp{L([\Ord]^{\aleph_0},\dot{A}^{\bool{B}})\models\phi(p,\dot{A}^{\bool{B}})}>0_{\bool{B}}$,
\item
$T\vdash \forall\bool{B}\Qp{L([\Ord]^{\aleph_0},\dot{A}^{\bool{B}})\models\phi(p,\dot{A}^{\bool{B}})}=1_{\bool{B}}$.
\end{itemize}
\end{theorem}
Notice that since $H_{\omega_1}\subseteq L([\Ord]^{\aleph_0})$,
via Theorem~\ref{thm:dual2omegaBnames} and natural generalizations of \cite[Lemma 25.25]{JECH}
establishing a correspondence between $\Sigma^1_{n+1}$-properties and $\Sigma_n$-properties 
over $H_{\omega_1}$, 
we obtain that for any complete boolean algebra
$\bool{B}$ and any $\Sigma^1_{n}$-predicate $R\subseteq (2^\omega)^n$ the map 
$x\mapsto[c_x]_G$ of $(2^\omega,R)$ into
$(C(\St(\bool{B},2^\omega),R^{\St(\bool{B})})$ is an elementary embedding. 
In particular the above theorem provides a first 
fully satisfactory answer to the question of whether the natural embeddings of $2^\omega$ into its boolean 
ultrapowers can be elementary: the answer is yes if we accept the existence of large cardinals!

The natural question to address now is whether we can step up this result also for uncountable 
$\lambda$.
If so, to which form?

\subsection{Topological maximality: Martin's maximum $\MM$}
Let us now address the quest for maximal forcing axioms from the topological direction. Specifically:
what is the largest class of partial orders $\Gamma$ for which we can posit
$\FA_{\aleph_1}(\Gamma)$?

%\Rnote*{reference?}{
Shelah proved
%} 
that $\FA_{\aleph_1}(P)$ fails for any $P$ which does not 
preserve stationary subsets of $\omega_1$. Nonetheless it cannot be decided in $\ZFC$ 
whether this is a necessary condition for a poset $P$ in order to have the failure of
$\FA_{\aleph_1}(P)$.
For example let $P$ be a forcing which shoots a club of ordertype $\omega_1$ through a
projectively stationary and costationary subset of $P_{\omega_1}(\omega_2)$
by selecting countable initial segments of this club: 
for all such $P$, it is provable in $\ZFC$ that $P$ 
preserve stationary subsets 
of $\omega_1$. However 
in $L$, $\FA_{\aleph_1}(P)$ fails for some such $P$ while in a
model of Martin's maximum $\MM$, $\FA_{\aleph_1}(P)$ holds for all such $P$.

The remarkable result of Foreman, Magidor, and Shelah~\cite{foreman_magidor_shelah} 
is that the above necessary condition
is consistently also a sufficient condition: it can be forced that 
$\FA_{\aleph_1}(P)$ holds if and only if $P$ is a forcing notion preserving all stationary subsets of $\omega_1$.
This axiom is known in the literature as Martin's maximum $\MM$.
In view of Theorem~\ref{thm:forcax-AC-BCT},
$\MM$ realizes a maximality property for forcing axioms: it can be seen as a maximal 
strengthening of the axiom of choice $\mathsf{AC}\rest{\omega_2}$ for $\aleph_1$-sized families of non-empty sets.
Can we strengthen this further? If so, in which form? It turns out that stronger and stronger forms of forcing axioms 
can be expressed in the language of categories and provide means to extend Woodin's generic absoluteness results
to third order arithmetic or more generally to larger and larger fragments of the set theoretic universe.

\subsection{Category forcings and category forcing axioms}

Assume $\Gamma$ is a class of complete boolean algebras and
$\rightarrow^\Theta$ is a family of complete homomorphisms between
elements of $\Gamma$
closed under composition and containing all identity maps. 
$(\Gamma,\rightarrow^\Theta)$ is the category
whose objects
are the complete boolean algebras in $\Gamma$ 
and whose arrows are given by complete homomorphisms
$i:\bool{B}\to\bool{Q}$ in $\rightarrow^\Theta$.
We call embeddings in $\rightarrow^\Theta$, $\Theta$-correct embeddings.
Notice that these categories immediately give rise to natural class pre-orders
associated with them, 
pre-orders whose elements are the complete boolean algebras in $\Gamma$
and whose order relation is given by the arrows in $\rightarrow^\Theta$ (i.e. $\bool{B}\leq_\Theta\bool{C}$ if
there exists $i:\bool{C}\to\bool{B}$ in $\rightarrow^\Theta$).
We denote these class partial orders by $(\Gamma,\leq_\Theta)$.

%Notice that if $\Theta\subseteq\Delta$ the incompatibility relation in 
Depending on the choice of $\Gamma$ and $\rightarrow^\Theta$ these partial orders can be trivial
(as forcing notions),
for example:
\begin{remark}\label{rem:trivOmega}
Assume $\Omega=\Omega_{\aleph_0}$ is the class of all complete boolean algebras and 
$\rightarrow^\Omega$ is the class of all complete embeddings, then any
two conditions in $(\Gamma,\leq_\Omega)$ are compatible, i.e.
$(\Gamma,\leq_\Omega)$ is forcing equivalent to the trivial partial order.
This is the case since for any pair of partial orders $P,Q$ and $X$ of size larger than
$2^{|P|+|Q|}$ there are 
complete injective homomorphisms of $\RO(P)$ and $\RO(Q)$ into the boolean completion of
$\Coll(\omega,X)$ (see~\cite[Thm A.0.7]{LARSON} and its following remark). 
These embeddings witness the compatibility of $\RO(P)$ with $\RO(Q)$.

On the other hand these class partial orders will in general be non-trivial:  let $\SSP$ be the class of stationary set preserving forcings. Then the Namba forcing shooting
a cofinal $\omega$-sequence on $\omega_2$ and $\Coll(\omega_1,\omega_2)$ are incompatible conditions
in $(\SSP,\leq_\Omega)$: any forcing notion absorbing both of them makes the cofinality of
$\omega_2^V$ at the same time of cofinality $\omega_1^V$ (using the generic filter for $\Coll(\omega_1,\omega_2)$)
and countable (using the generic filter for Namba forcing); this means that this forcing must collapse $\omega_1^V$ to become a countable ordinal, hence cannot be stationary set preserving.
\end{remark}

\subsubsection*{Forcing axioms as density properties of category forcings}
The following results are among the main reasons to analyze in more
detail these types of class forcings:

\begin{theorem}[Woodin, Thm. 2.53~\cite{woodinBOOK}]\label{cor:MM++den}
Assume there are class many supercompact cardinals. Then the following are equivalent for any complete
cba $\bool{B}$ and cardinal $\kappa$:
\begin{enumerate}
\item
$\FA_\kappa(\bool{B})$;
\item
there is a complete homomorphism of $\bool{B}$ into a presaturated tower inducing a generic 
ultrapower embedding with critical point $\kappa^+$.
\end{enumerate}
\end{theorem}

\begin{theorem}[V. Thm. 2.12~\cite{VIAMMREV}]\label{cor:MM++den2}
Assume there are class many supercompact cardinals. Then the following are equivalent:
\begin{enumerate}
\item
$\MM^{++}$;
\item
the class of presaturated normal towers is dense in $(\SSP,\leq_{\SSP})$.
\end{enumerate}
\end{theorem}

It is not in the scope of this paper to delve into the definition and properties of presaturated tower forcings and of 
the axiom $\MM^{++}$.
Let us just remark the following two facts: 
\begin{itemize}
\item
$\MM^{++}$ is a natural strengthening of Martin's maximum
whose consistency is proved by exactly the same methods producing a model of Martin's maximum.
\item
A presaturated tower $\tow{T}$ inducing a generic ultrapower embedding with critical point $\kappa^+$
is such that whenever $G$ is $V$-generic for $\tow{T}$ we have that
%\begin{enumerate}
%\item \label{prop:denscf1}
\begin{equation} \label{prop:denscf1}
H_{\kappa^+}^V\prec H_{\kappa^+}^{V[G]}. 
\end{equation}
%\item \label{prop:denscf2}
%there exists an elementary $j: L(\Ord^\kappa)^V\to L(\Ord^\kappa)^{V[G]}$ with critical point $\kappa^+$.
%\end{enumerate}
\end{itemize}
In particular the above theorems show that forcing axioms can be also stated as density properties of class 
partial orders. 
Below we will describe assumptions  $\mathsf{AX}(\Gamma,\kappa)$
yielding a dense class of
forcings in $(\Gamma,\leq_\Gamma)$
whose
generic extensions satisfy (\ref{prop:denscf1}), and producing generic absoluteness results.
%On the other hand if 
%$\mathsf{AX}_2(\Gamma,\kappa)$ is an axiom which 
%yields a dense class of forcings in $(\Gamma,\leq_\Gamma)$ whose
%generic extensions satisfy item \label{prop:denscf2}, we obtain a generic absoluteness result for $L(\Ord^\kappa)$
%for all
%forcings in $\Gamma$ preserving this axiom.
%Let us handle with some care the case of why $\mathsf{AX}_1(\Gamma,\kappa)$ yields generic absoluteness for 
%$H_{\kappa^+}$ and  give some hints on how to get that $\mathsf{AX}_2(\Gamma,\kappa)$ yields
%a generic absoluteness result for $L(\Ord^\kappa)$.
We refer the reader to \cite{VIAAUDSTEBOOK,VIAAUD14,VIAMM+++} for details. 

\subsection{Iterated resurrection axioms and generic absoluteness for $H_{\kappa^+}$}

The results and ideas of this subsection
expand on the seminal work of Hamkins and Johnstone~\cite{HAMJOH13} on resurrection axioms.
\begin{definition}
		Let $\Gamma$ be a definable class of complete Boolean algebras closed under two-step iterations.		
		The \emph{cardinal preservation degree} $\cpd(\Gamma)$ of $\Gamma$ is the largest cardinal 
		$\kappa$ such that every $\bool{B} \in \Gamma$ forces that every cardinal $\nu \leq \kappa$ is still a 
		cardinal in $V^\bool{B}$. If all cardinals are preserved by $\Gamma$, we say that $\cpd(\Gamma) = \infty$.

		The \emph{distributivity degree} $\dd(\Gamma)$ of $\Gamma$ is the largest cardinal 
		$\kappa$ such that every $\bool{B} \in \Gamma$ is ${<}\kappa$-distributive.
	\end{definition}

	We remark that the supremum of the cardinals preserved by $\Gamma$ is preserved by $\Gamma$, and the same holds for the property of being ${<}\kappa$ distributive. Furthermore, $\dd(\Gamma) \leq \cpd(\Gamma)$ and $\dd(\Gamma)\neq\infty$ whenever $\Gamma$ is non trivial (i.e., it contains a Boolean algebra that is not forcing equivalent to the trivial Boolean algebra).
	Moreover $\dd(\Gamma)=\cpd(\Gamma)$ whenever $\Gamma$ is closed under two-step iterations and contains the class of ${<}\cpd(\Gamma)$-closed posets.
	
	\begin{definition} \label{def:gamma_gamma}
		Let $\Gamma$ be a definable class of complete Boolean algebras. We let 
		$\gamma = \gamma_\Gamma = \cpd(\Gamma)$.
%		\begin{itemize}
%			\item $\cpd(\Gamma)$ if $\cpd(\Gamma) < \infty$;
%			\item $\dd(\Gamma)$ otherwise.
%		\end{itemize}
	\end{definition}
	
	For example, $\gamma = \omega$ if $\Gamma $ is the class of all posets, while for axiom-$A$, proper, $\SP$, $\SSP$ we have that $\gamma = \omega_1$, and for ${<}\kappa-$closed we have that $\gamma = \kappa$.
	
	We aim to isolate for each cardinal $\gamma$ classes of forcings $\Delta_\gamma$ and 
	axioms $\AX(\Delta_\gamma)$ such that:
	\begin{enumerate}
		\item \label{aim:resurrection-1}
		$\gamma=\cpd(\Delta_\gamma)$ and assuming certain large cardinal axioms, the family of 
		$\bool{B}\in \Delta_\gamma$ which force $\AX(\Delta_\gamma)$ is dense in 
		$(\Delta_\gamma,\leq_{\Delta_\gamma})$;
		\item \label{aim:resurrection-2}
		$\AX(\Delta_\gamma)$ gives generic absoluteness for the theory with parameters of 
		$H_{\gamma^+}$ with respect to all forcings in $\Delta_\gamma$ which preserve $\AX(\Delta_\gamma)$;
		\item
		the axioms $\AX(\Delta_\gamma)$ are mutually compatible for the largest possible family of cardinals 
		$\gamma$ simultaneously;
		\item
		the classes $\Delta_\gamma$ are the largest possible for which the axioms 
		$\AX(\Delta_\gamma)$ can possibly be consistent.
	\end{enumerate}

	Towards this aim remark the following:
	\begin{itemize} 
		\item
		$\dd(\Gamma)$ is the least possible cardinal $\gamma$ such that $\AX(\Gamma)$ is a 
		non-trivial axiom asserting generic absoluteness for the theory of $H_{\gamma^+}$ with parameters. 
		In fact, $H_{\dd(\Gamma)}$ is never changed by forcings in $\Gamma$. 
		\item
		$\cpd(\Gamma)$ is the largest possible cardinal $\gamma$ for which an axiom $\AX(\Gamma)$ 
		as above can grant generic absoluteness with respect to $\Gamma$ for the theory of $H_{\gamma^+}$ 
		with parameters. To see this, let $\Gamma$ be such that $\cpd(\Gamma) = \gamma $ and assume 
		towards a contradiction that there is an axiom $\AX(\Gamma)$ yielding generic absoluteness with 
		respect to $\Gamma$ for the theory with parameters of $H_{\lambda}$ with $\lambda>\gamma^+$.

		Assume that $\AX(\Gamma)$ holds in $V$. Since $\cpd(\Gamma) = \gamma$, there exists a 
		$\bool{B} \in \Gamma$ which collapses $\gamma^+$. Let $\bool{C} \leq_\Gamma \bool{B}$ be 
		obtained by property (\ref{aim:resurrection-1}) above for
		$\Gamma=\Delta_\gamma$, so that $\AX(\Gamma)$ holds in $V^\bool{C}$, and remark that 
		$\gamma^+$ cannot be a cardinal in $V^\bool{C}$ as well. 
		Then $\gamma^+$ is a cardinal in $H_{\lambda}$ and not in $H_{\lambda}^\bool{C}$, 
		witnessing failure of generic absoluteness and contradicting property 
		(\ref{aim:resurrection-2}) for $\AX(\Gamma)$.
	\end{itemize}

	We argue that there are axioms $\RA_\omega(\Gamma)$
	satisfying the first two of the above requirements, and which 
	are consistent for a variety of forcing classes $\Gamma$.
	These axioms also provide natural examples for the last two requirements. 
	We will come back later on with philosophical considerations outlining why the last 
	two requirements are also natural.
	We can prove the consistency of $\RA_\omega(\Gamma)$ for forcing classes which are definable in 
	G\"odel-Bernays set theory with classes $\NBG$, closed under two-step iterations,
	weakly iterable (a technical definition asserting that most set sized descending sequences in
	$\leq_\Gamma$ have lower bounds in $\Gamma$, see \cite{VIAAUD14} or \cite{VIAAUDSTEBOOK} 
	for details), and contain
	all the $<\cpd(\Gamma)$-closed forcings.
%	To obtain the consistency of $\CFA(\Gamma)$ we will need a 
%	further crucial property of $\Gamma$: that of containing a dense class of $\Gamma$-rigid elements, i.e.
%	$\bool{B}\in\Gamma$ addmitting at most one $i:\bool{B}\to\bool{C}$ witnessing $\bool{B}\leq_\Gamma\bool{C}$
%	for all $\bool{C}\in\Gamma$.
	
	The axioms $\RA_\alpha(\Gamma)$ for $\alpha$ an ordinal can be formulated in 
	the Morse Kelley axiomatization of set theory $\MK$ as follows:
	\begin{definition}\label{prop:radef}
		Given an ordinal $\alpha$ and a 
		definable\footnote{$\Gamma$ must be definable 
		by a formula with no class quantifier and no class parameter to be on the safe side fwith respect to the definability issues regarding the iterated resurrection axioms 
		raised by the remark right after this definition. All usual classes of forcings such as proper, semiproper,
		stationary set preserving, $<\kappa$-closed, etc.... are definable by formulae satisfying these restrictions.} 
		class of forcings $\Gamma$ closed under two-steps
		iterations, the axiom
			$\RA_\alpha(\Gamma)$ holds if for all $\beta < \alpha$ the class
			\[
			\bp{\bool{B} \in \Gamma: ~~ H_{\gamma^+} \prec H^\bool{B}_{\gamma^+} \wedge 
			V^\bool{B} \models \RA_\beta(\Gamma)}
			\]
			is dense in $\cp{\Gamma, \leq_\Gamma}$ (where $\gamma=\gamma_\Gamma$).
			
			$\RA_\Ord(\Gamma)$ holds if $\RA_\alpha(\Gamma)$ holds for all $\alpha$.
		\end{definition}

\begin{remark}
The above definition can be properly formalized in $\MK$ (but most likely not in $\ZFC$ if $\alpha$ is infinite).
The problem is the following: 
the axioms $\RA_\alpha(\Gamma)$ can be formulated only by means of a transfinite recursion over
a well-founded relation which is not set-like. 
It is a delicate matter to argue that this transfinite recursion can be carried out.
\cite{VIAAUD14} shows that this is the case if the base theory is $\MK$.
%a legitimate definition in $\MK$, we refer the reader to \cite{VIAAUD14} for details.
\end{remark}

The axiom $\RA_\omega(\Gamma)$ yields generic absoluteness by the following elementary argument: 
\begin{theorem} \label{thm:absoluteness}
		Suppose $n \in \omega$, $\Gamma$ is well behaved, 
		$\RA_n(\Gamma)$ holds, and $\bool{B} \in \Gamma$ forces $\RA_n(\Gamma)$. 
		Then $H_{\gamma^+} \prec_n H_{\gamma^+}^\bool{B}$ (where $\gamma=\gamma_\Gamma$).
	\end{theorem}
	\begin{proof}
		We proceed by induction on $n$. Since $\gamma^+ \leq (\gamma^+)^{V^\bool{B}}$, $H_{\gamma^+} \subseteq H_{\gamma^+}^\bool{B}$ and the thesis holds for $n = 0$ by the fact that for all transitive structures $M$, $N$, if $M \subset N$ then $M \prec_0 N$.
		 Suppose now that $n > 0$, and fix $G$ $V$-generic for $\bool{B}$. By $\RA_n(\Gamma)$, let $\bool{C} \in V[G]$ be such that whenever $H$ is $V[G]$-generic for $\bool{C}$, $V[G \ast H] \models \RA_{n-1}(\Gamma)$ and $H_{\gamma^+}^V \prec H_{\gamma^+}^{V[G \ast H]}$. 
		 Hence we have the following diagram:
		\[
			\begin{tikzpicture}[xscale=1.5,yscale=-1.2]
				\node (A0_0) at (0, 0) {$H_{\gamma^+}^V$};
				\node (A0_2) at (2, 0) {$H_{\gamma^+}^{V[G \ast H]}$};
				\node (A1_1) at (1, 1) {$H_{\gamma^+}^{V[G]}$};
				\path (A0_0) edge [->]node [auto] {$\scriptstyle{\Sigma_\omega}$} (A0_2);
				\path (A1_1) edge [->]node [auto,swap] {$\scriptstyle{\Sigma_{n-1}}$} (A0_2);
				\path (A0_0) edge [->]node [auto,swap] {$\scriptstyle{\Sigma_{n-1}}$} (A1_1);
			\end{tikzpicture}
		\]
		obtained by inductive hypothesis applied both on $V$, $V[G]$ and on $V[G]$, $V[G \ast H]$ since in all those classes $\RA_{n-1}(\Gamma)$ holds.

		Let $\phi \equiv \exists x \psi(x)$ be any $\Sigma_{n}$ formula with parameters in $H_{{\gamma^+}}^V$. First suppose that $\phi$ holds in $V$, and fix $\bar{x} \in V$ such that $\psi(\bar{x})$ holds. Since $H_{\gamma^+}^V \prec_{n-1} H_{\gamma^+}^{V[G]}$ and $\psi$ is $\Pi_{n-1}$, it follows that $\psi(\bar{x})$ holds in $V[G]$ hence so does $\phi$.
		Now suppose that $\phi$ holds in $V[G]$ as witnessed by $\bar{x} \in V[G]$. Since $H_{\gamma^+}^{V[G]} \prec_{n-1} H_{\gamma^+}^{V[G \ast H]}$ it follows that $\psi(\bar{x})$ holds in $V[G \ast H]$, hence so does $\phi$. Since $H_{\gamma^+}^V \prec H_{\gamma^+}^{V[G \ast H]}$, the formula $\phi$ holds also in $V$ concluding the proof.
	\end{proof}

	\begin{corollary} \label{cor:absoluteness}
		Assume $\Gamma$ is closed under two-steps iterations and contains the $<\cpd(\Gamma)$-closed forcings.
		If $\RA_\omega(\Gamma)$ holds, and $\bool{B} \in \Gamma$ forces $\RA_\omega(\Gamma)$, 
		then  $H_{\gamma^+} \prec H_{\gamma^+}^\bool{B}$ (where $\gamma=\gamma_\Gamma$).
	\end{corollary}
	
	Regarding the consistency of the axioms $\RA_\omega(\Gamma)$ we have the following:
		\begin{proposition} \label{prop:woodinrall}
		Assume there are class-many Woodin cardinals. Then $\RA_\Ord(\Omega)$ holds.
	\end{proposition}

	\begin{theorem}
		$\RA_1(\Gamma)$ implies $H_{\gamma^+} \prec_1 V^\bool{B}$ for all $\bool{B} \in \Gamma$, hence it 
		is a strengthening of the bounded forcing axiom\footnote{The bounded forcing axiom 
		$\BFA_\gamma(\Gamma)$ asserts that $H_{\gamma^+}\prec_1 V^{\bool{B}}$
		for all $\bool{B}\in\Gamma$.} $\BFA_\gamma(\Gamma)$ (where $\gamma=\gamma_\Gamma$).
	\end{theorem}

	\begin{theorem}[\cite{VIAAUD14}]
		Assume there is a super huge cardinal.\footnote{A cardinal $\kappa$ is \emph{super huge} iff for every ordinal $\alpha$ there exists an elementary embedding $j: V \to M\subseteq V$ with $\crit(j) = \kappa$, $j(\kappa) > \alpha$ and ${}^{j(\kappa)}M \subseteq M$.} 
		
		Then $\RA_{\Ord}(\SSP)+\MM^{++}$ and $\RA_{\Ord}(\mathsf{proper})+\PFA^{++}$ are consistent.
		
		For the consistency of $\RA_{\Ord}(\mathsf{proper})$ a Mahlo cardinal suffices.
		
		Moreover
		it is also consistent relative to a Mahlo cardinal that $\RA_{\Ord}(\Gamma_\kappa)$ holds 
		simultaneously for all cardinals $\kappa$ (where $\Gamma_\kappa$ is the class of $<\kappa$-closed
		forcings)\footnote{It is also consistent the following: 
		\[
		\RA_\Ord(\Omega_{\aleph_0})+ 
		\RA_\Ord(\SSP)+\forall\kappa>\omega_1\,\RA_{\Ord}(\Gamma_\kappa)
		\]}.
	\end{theorem}

In this regard the axioms $\RA_\alpha(\Gamma)$ for $\Gamma\supseteq\Gamma_\kappa$ ($\Gamma_\kappa$ being the class of $<\kappa$-closed forcings)
appear to be natural companions of the axiom of choice, while the axioms $\RA_\Ord(\Omega)$ and 
$\RA_\Ord(\SSP)+\MM$ %(or $\MM^{+++}$) 
are natural maximal strengthenings 
of the axiom of choice at the levels $\omega$ and $\omega_1$.
Hence it is in our opinion natural to try to isolate classes of forcings $\Delta_\kappa$ as $\kappa$ 
ranges among the cardinals such that:
\begin{enumerate}%[(a)]
\item $\kappa=\cpd(\Delta_\kappa)$ for all $\kappa$.
\item $\Delta_\kappa\supseteq \Gamma_\kappa$ for all $\kappa$.
\item $\FA_\kappa(\Delta_\kappa)$ and $\RA_\omega(\Delta_\kappa)$ are simultaneously consistent for all $\kappa$.
\item For all cardinals $\kappa$, $\Delta_\kappa$ is the largest possible 
$\Gamma$ with $\cpd(\Gamma)=\kappa$ 
for which  $\FA_\kappa(\Delta_\kappa)$ and $\RA_\omega(\Delta_\kappa)$ are simultaneously 
consistent (and if possible for all $\kappa$ simultaneously). 
\end{enumerate}
Compare the above requests with requirements (3) and (4) in the discussion motivating the introduction of 
the iterated resurrection axioms
on page~\pageref{aim:resurrection-1}. In this regard it appears that we have now a completely satisfactory 
answer about what $\Delta_{\omega}$ and $\Delta_{\omega_1}$ are: i.e., respectively the class of \emph{all} 
forcing notions and the class of all $\SSP$-forcing notions.

\subsection{Boosting Woodin's absoluteness to $L(\Ord^\kappa)$: the axioms $\mathsf{CFA}(\Gamma)$}

We gave detailed arguments leading us to axioms which can be stated as density properties of certain category forcings
and yielding generic absoluteness results for the theory of $H_{\kappa^+}$ for various cardinals $\kappa$.
Exploring Woodin's proof for the generic absoluteness of the Chang model $L(\Ord^\omega)$ one can get an even stronger type of category forcing axiom yielding generic absoluteness results for the 
Chang models $L(\Ord^\kappa)$. The best result we can currently produce is the following
(we refer the interested reader to ~\cite{VIAASP,VIAAUDSTEBOOK,VIAMM+++} for details):

\begin{theorem}\label{thm:mainthm}
Let $\Gamma$ be a $\kappa$-suitable class of forcings\footnote{This is a lenghty and technical definition; roughly
it requires that:
\begin{itemize}
\item 
$\Gamma$ is closed under two-step iterations, and contains all the $<\kappa$-closed posets
(where $\kappa=\cpd(\Gamma)$), 
\item
there is an iteration 
theorem granting that all set sized iterations of posets in
$\Gamma$ has a limit in $\Gamma$, 
\item
$\Gamma$ is defined by a syntactically simple formula (i.e. 
$\Sigma_2$ in the Levy hierarchy of formulae), 
\item 
$\Gamma$ has a dense set of $\Gamma$-rigid elements 
(i.e. the $\bool{B}\in \Gamma$ admitting at most one $i:\bool{B}\to \bool{C}$ witnessing that
$\bool{C}\leq_\Gamma\bool{B}$ for all $\bool{C}\in\Gamma$ form a dense subclass of $\Gamma$).
\end{itemize}}.

Let $\MK^*$ stand for\footnote{In $\MK$ one can define the club filter on the class $\Ord$, 
hence the notion of stationarity for classes of ordinals makes sense.
%$\delta$ is a $\Sigma_2$-reflecting cardinal if it is inaccessible and
%for all
%formulae $\phi(x)$ (in one free variable and with quantifiers ranging only over sets) 
%and $A\in V_\delta$, there exists $\alpha$ such that $V_\alpha\models\phi(A)$
%if and only if there exists an $\alpha<\delta$ with this property.
} 
\[
\MK+\text{ there are stationarily many 
inaccessible cardinals. }
\]
There is an axiom\footnote{$\mathsf{CFA}(\Gamma)$ can be formulated as a density property of the class
forcing $(\Gamma,\leq_\Gamma)$.} $\mathsf{CFA}(\Gamma)$  
which implies $\FA_\kappa(\Gamma)$ as well as $\RA_\Ord(\Gamma)$ and is such that for
any $T^*$  extending 
\[
\MK^*+\mathsf{CFA}(\Gamma)+\kappa\text{ is a regular
cardinal }+S\subset\kappa,
\]
and for any formula
$\phi(S)$, the following are equivalent:
\begin{enumerate}
\item
$T^*\vdash [L(\Ord^{\kappa})\models\phi(S)]$,
\item
$T^*$ proves that for some forcing $\bool{B}\in\Gamma$ 
\[
\Qp{\mathsf{CFA}(\Gamma)}_{\bool{B}}=\Qp{L(\Ord^{\kappa})\models\phi(S)}_{\bool{B}}=1_{\bool{B}}.
\]
\end{enumerate}
\end{theorem}
We also have that
\begin{theorem}[\cite{VIAASP,VIAAUDSTEBOOK}]
Assume $\Gamma$ is $\kappa$-suitable. Then $\mathsf{CFA}(\Gamma)$ is consistent relative to the 
existence of a $2$-superhuge cardinal\footnote{A cardinal $\kappa$ is  $2$-superhuge if it is supercompact and 
this can be witnessed
by $2$-huge embeddings.}.
\end{theorem}
While the definition of $\kappa$-suitable $\Gamma$ is rather delicate, it can be shown that
many interesting classes are $\omega_1$-suitable, among others: 
proper, semiproper, $\omega^\omega$-bounding and (semi)proper, preserving a suslin tree and (semi)proper.
\cite{VIAASP} contains a detailed list of classes which are $\omega_1$-suitable.
It is not known whether there can be $\kappa$-suitable classes $\Gamma$ for some $\kappa>\omega_1$.

\section{Some open questions}

Here is a list of questions for which we do not have many clues.....
\begin{enumerate}
\item
What are the $\Gamma$ which are $\kappa$-suitable for a given cardinal $\kappa>\aleph_1$
(i.e. such that $\mathsf{CFA}(\Gamma)$ is consistent)?
\item
Do they even exist for $\kappa>\aleph_1$?
\item
In case they do exist for some $\kappa>\aleph_1$, do we always have a unique maximal $\Gamma$ such
that $\cpd(\Gamma)=\kappa$ as is the case for $\kappa=\aleph_0$ or $\kappa=\aleph_1$?
\end{enumerate}
Any  
interesting iteration theorem for a class $\Gamma\supseteq \Gamma_{\omega_2}$ closed under two-step 
iterations can be used to prove that $\RA_\Ord(\Gamma)$ is consistent relative to 
suitable large cardinal assumptions and 
freezes the theory of $H_{\omega_3}$ with respect to forcings in $\Gamma$ preserving $\RA_\omega(\Gamma)$
(see~\cite{VIAAUD14}).
It is nonetheless still a mystery which classes $\Gamma\supseteq\Gamma_{\omega_2}$ can give us
a nice iteration theorem, even if the recent works by Neeman, Asper\`o, Krueger,
Mota, Velickovic and others are starting to shed some light on this problem 
(see among others ~\cite{krueger:forcingclubaleph2,krueger:mota:stronglyproper,neeman:forcingaleph2}).
%similarly \cite{VIAAUDSTEBOOK} shows that if $\Gamma$ is $\kappa$-suitable then $\mathsf{CFA}(\Gamma)$
%is consistent and yields generic absoluteness for the Chang model $L(\Ord^\kappa)$.

We can dare to be more ambitious and replicate the above type of issue at a 
much higher level of the set theoretic hierarchy.
There is a growing set of results regarding the first-order theory of 
$L(V_{\lambda+1})$ assuming $\lambda$ is a very large cardinal
(i.e., for example admitting an elementary $j: L(V_{\lambda+1})\to L(V_{\lambda+1})$ 
with critical point smaller than $\lambda$, see for example~\cite{dimonte:I0GCH,dimonte:I0,woodin:beyondI0}). 
It appears that large fragments of this theory are generically invariant with respect to a great variety of forcings.

\begin{quote}
Assume $j:L(V_{\lambda+1})\to L(V_{\lambda+1})$ is
elementary with critical point smaller than $\lambda$ .
Can any of the results presented in this paper be of any use in the study of which type of 
generic absoluteness results may hold at the level of $L(V_{\lambda+1})$?
\end{quote}

The reader is referred to \cite{VIAASP,VIAAUDSTEBOOK,VIAAUD14,VIAMM+++,VIAMMREV} for further
examinations of these topics.

\subsubsection*{Acknowledgements}
This paper owes much of its clarity to the suggestions of Raphael Carroy, 
and takes advantage of several several fruitful discussions we shared on the material presented here.
I wish to thank Jeffrey Bergflak for his many useful comments.	

\smallskip

\noindent The author acknowledges support from: 
Kurt G\"odel Research Prize Fellowship 2010,
PRIN grant 2012: Logica, modelli e insiemi, 
San Paolo Junior PI grant 2012.

	\bibliographystyle{amsplain}
	\bibliography{Biblio}
\end{document}